\makeatother
\documentclass[12 pt]{amsart}
\title{Sparsity of rational points on torsion level covers \\of Hilbert modular varieties}
\author{Soheil Memariansorkhabi}\address{Department of Mathematics, Purdue University,
West Lafayette, IN 47907, USA}

\address{Department of Mathematics, University of Toronto,
Toronto, Canada}

\email{smemaria@purdue.edu}
\email{soheil.memarian@alumni.utoronto.ca}

\makeatletter
\renewcommand{\@setauthors}{%
  \begingroup
  \def\thanks{\protect\thanks@warning}%
  \trivlist
  \centering\footnotesize \@topsep30\p@\relax
  \advance\@topsep by -\baselineskip
  \item\relax
  \author@andify\authors
  \def\\{\protect\linebreak}%
  \authors
  \endtrivlist
  \endgroup
}

\usepackage[utf8]{inputenc}

\usepackage{amsfonts}

\usepackage{mathrsfs}
\usepackage[utf8]{inputenc}
\usepackage{amsmath, amsfonts, amsthm}
\usepackage{amssymb}
\usepackage{fullpage}
\usepackage{mathtools}
\usepackage{enumitem}
\usepackage{amsmath}
\usepackage{hyperref}
\usepackage[dvipsnames]{xcolor}

\hypersetup{
    colorlinks=true,
    citecolor=blue,
    linkcolor=magenta,
    filecolor=magenta,      
    urlcolor= black,
    }

\newtheorem{theorem}{Theorem}[section]
\newtheorem{Lemma}[theorem]{Lemma}
\newtheorem{Proposition}[theorem]{Proposition}
\newtheorem{Corollary}[theorem]{Corollary}
\theoremstyle{definition}
\newtheorem{definition}[theorem]{Definition}

\newtheorem{Remark}[theorem]{Remark}
\newtheorem*{theorem*}{Theorem}
\newtheorem*{Remark*}{Remark}
\newtheorem*{Conjecture*}{Conjecture}

\newcommand{\Q}{\mathbb{Q}} % rationals
\newcommand{\C}{\mathbb{C}} % complex

\newcommand{\Proj}{\mathbb{P}}

\newcommand{\bi}{\big}

\newcommand{\Xb}{\overline{X_{1}}}

\newcommand{\Nm}{\operatorname{Nm}}

\newcommand{\SL}{\operatorname{SL}}
\newcommand{\Hu}{\mathbb{H}}

\usepackage{mathtools}
\counterwithin{equation}{section}

\usepackage{etoolbox}

\makeatletter

\patchcmd{\@settitle}
  {\uppercasenonmath\@title}
  {}
  {}
  {}

\patchcmd{\section}
  {\scshape}
  {\bfseries}
  {}
  {}
\makeatother

\begin{document}
\maketitle

\begin{abstract}
Let $F$ be a totally real field of degree $n$ and discriminant $\Delta_F$, and let $X_1(\eta)$ be the cover of the Hilbert modular variety parametrizing abelian varieties with real multiplication by $\mathcal O_F$, together with a torsion point having annihilator $\eta$. Let $L=K_{\overline X_1(\eta)}+D$ be the log-canonical bundle on a smooth toroidal compactification, and let \(H_L\) be an associated
multiplicative height. We prove that rational points on $X_1(\eta)$ become sparser as $|\mathrm{Nm}(\eta)|\to\infty$, with $(\eta,\Delta_F)=1$. More
precisely, for every number field $K$, set
  $$ N_{\eta,K}(B)=\#\{x\in X_1(\eta)(K):H_L(x)\leq B\}. $$
 If $|\mathrm{Nm}(\eta)|\ge 5^n$ and $(\eta,\Delta_F)=1$, we prove
$$ \limsup_{B\to\infty}\frac{\log\max\{1,N_{\eta,K}(B)\}}{\log B} \leq\delta_{\eta,K,n},\qquad \delta_{\eta,K,n}\ll_{[K:\mathbb Q],n}|\mathrm{Nm}(\eta)|^{-1/(2n)}. $$
  In particular, $\delta_{\eta,K,n}\to0$ uniformly when $n$ and $[K:\mathbb Q]$ are bounded and $|\mathrm{Nm}(\eta)|\to\infty$.
  
  The main geometric result is a uniform lower bound, growing with the level, for the log-canonical degree of subvarieties of $X_1(\eta)$. Combining this estimate with recent progress derived from determinant-method, due to Ellenberg--Lawrence--Venkatesh and Brunebarbe--Maculan, we obtain the sparsity result above.
  We also prove that, for sufficiently large $|\mathrm{Nm}(\eta)|$, every subvariety of $X_1(\eta)$ is of general type, and establish a higher-dimensional generalization of the geometric torsion theorem of Bakker--Tsimerman. Namely, for a family of abelian varieties with real multiplication over a quasi-projective base of arbitrary dimension, we bound the torsion subgroup of its Mordell--Weil group in terms of the canonical volume of the base, uniformly in the totally real multiplication field of fixed degree.\\
  
\end{abstract}

\section*{Introduction}
Let $E$ be an elliptic curve defined over a number field $K.$ When $K=\Q,$ by the celebrated theorem of Mazur \cite{mazur1978rational} the torsion part $E(K)_{tor}$ is uniformly bounded. This theorem is generalized to arbitrary number field by Kamienny and Merel:

\begin{theorem*}(\cite{merel1996bornes,kamienny1992torsion}) Let $K$ be a number field of degree $d$ over $\Q.$ Then, there exists $N=N(d)$ such that for all elliptic curves $E/K$ we have that 
    \[
        E(K)_{tor}\subset E(K)[N]. 
    \]
\end{theorem*}
The arithmetic torsion conjecture predicts the same behavior for $n$-dimensional abelian varieties over number fields: 

\begin{Conjecture*}[Arithmetic Torsion Conjecture]
For every pair of integers \(d,n\geq 1\), there exists an integer
\(N=N(n,d)\) such that, for every number field \(K\) with
\([K:\mathbb{Q}]\leq d\) and every \(n\)-dimensional abelian variety \(A\)
over \(K\),
\[
A(K)_{\mathrm{tor}}\subseteq A(K)[N].
\]
\end{Conjecture*}
Beyond the case of elliptic curves, this conjecture remains widely open. Its
geometric analogue over function fields of curves has been studied in
\cite{cadoret2011weak,cadoret2012uniform,bakker2018geometric} and, more
recently, in \cite{looperYap2026,gaoGu2026}. It is natural to seek an analogue of the geometric torsion conjecture over
function fields of higher-dimensional varieties. As explained in \cite[Appendix A]{cadoret2011weak}, once the function field of the base is fixed, a torsion bound depending on that function field follows from the curve case by taking suitable hyperplane sections. Our aim is different. We seek a bound that is uniform as the base varies and depends on the base only through an intrinsic  birational invariant.

Fix a totally
real field \(F/\mathbb{Q}\) of degree \(n\) and absolute discriminant \(\Delta_F\). By an abelian variety with real
multiplication by \(\mathcal{O}_F\), we mean an \(n\)-dimensional abelian
variety \(A\) equipped with an injection $\mathcal{O}_F\hookrightarrow\operatorname{End}(A).$
We consider such abelian varieties over the function fields of
quasi-projective varieties \(Z\) of arbitrary dimension and bound the torsion
subgroups of their Mordell--Weil groups in terms of the canonical volume of
\(Z\). Here, \emph{the canonical volume} of \(Z\) is the volume of the
canonical bundle of any smooth projective variety \(\overline{Z}\) birational
to \(Z\); that is,
\[
\widetilde{\operatorname{vol}}_{Z}
:=
\operatorname{vol}(K_{\overline{Z}})
=
\limsup_{b\to\infty}
\frac{h^0(\overline{Z},bK_{\overline{Z}})}{b^m/m!},
\]
 where \(m=\dim Z\). This quantity is independent of the chosen smooth projective model. In particular,  when
\(Z\) is a curve of geometric genus \(g\), it specializes to
\(
\widetilde{\operatorname{vol}}_{Z}
=
\max\{2g-2,0\}.
\)

Our first geometric result
generalizes \cite[Theorem~A]{bakker2018geometric}:
\renewcommand*{\thetheorem}{\Alph{theorem}}
\begin{theorem}\label{Canonical vol}(Theorem \ref{Main})
For every real number \(v\geq 0\) and positive integer \(n\), there exists
an integer \(N=N(n,v)\) with the following property. For any totally real field
\(F/\mathbb Q\) of degree \(n\), any complex positive-dimensional integral
quasi-projective variety \(Z\) with canonical volume at most \(v\), and
any abelian scheme \(A\to Z\) without isotrivial directions, of relative
dimension \(n\) and with real multiplication by \(\mathcal O_F\), the
torsion part of the Mordell--Weil group is uniformly bounded:
\[
A(Z)_{\mathrm{tor}}\subseteq A(Z)[N].
\]
Explicitly, as explained in Remark \ref{explicit N}, one may take
\[N(n,v):=
\left(
\left\lceil
\left(4\pi n\max\{1,v\}\right)^{2n}
\right\rceil-1
\right)!.\]
\end{theorem}

For a fixed totally real field \(F\), there is a Hilbert modular variety
\(X(1)\) parametrizing abelian varieties with real multiplication by
\(\mathcal O_F\).\footnote{For the polarization convention, see Remark~\ref{polarization-data}.} See Section~\ref{hilbert} for more details. A family
\(A\to Z\) of such abelian varieties determines a moduli map
\[
\mu_A\colon Z\longrightarrow X(1).
\]
We say that \(A\to Z\) has \emph{no isotrivial directions} if
\(\mu_A\) is generically finite onto its image. When \(Z\) is a curve,
this condition is equivalent to non-isotriviality. In higher dimensions,
it excludes families pulled back from a lower-dimensional base. This restriction is essential for the canonical volume bound. Indeed, if \(B\to C\) is any non-isotrivial family over a curve, then its pullback to \(C\times\mathbb P^1\) remains non-isotrivial, but \(C\times\mathbb P^1\) has canonical volume zero.

\begin{Remark*}[\textbf{Non-reduction to the curve case}]
Theorem~\ref{Canonical vol} does not follow from the curve case.
Indeed, restricting the family to a curve \(C\subset Z\), one may obtain a bound in
terms of the genus or gonality of \(C\). To recover
Theorem~\ref{Canonical vol}, one would need the resulting lower bounds for
these curve invariants to imply a lower bound for the
canonical volume of \(Z\). There is no general way to pass from such curvewise bounds to a lower bound for the canonical volume of \(Z\). For example, let \(Y\) be a very general abelian
variety of dimension \(m\geq 4\). Then every curve
\(C\subset Y\) has gonality at least
\(\left\lfloor\frac{m}{2}\right\rfloor+2\)
by \cite[Corollary~4.7]{martin2020gonality}. Nevertheless,
\(Y\) has canonical volume zero. This shows that the
minimum gonality of curves can be arbitrarily large while the canonical
volume of the ambient variety remains zero.
\end{Remark*}
 For any ideal \(\eta\subset\mathcal O_F\), the torsion level cover
\(X_1(\eta)\) of $X(1)$ parametrizes abelian varieties with real multiplication
by \(\mathcal O_F\), together with a point whose annihilator is
\(\eta\). Accordingly, a torsion section \(x\in A(Z)\) with
\(
\operatorname{Ann}_{\mathcal O_F}(x)=\eta
\)
induces a lift of the moduli map
\[
\mu_{A,x}\colon Z\longrightarrow X_1(\eta).
\]
By a subvariety of $X_1(\eta)$ we mean a positive-dimensional closed,
reduced, irreducible subvariety of a smooth projective toroidal
compactification $\Xb(\eta)$ that is not contained in the boundary.
Such a subvariety need not be smooth; it is of general type if a smooth
projective model has positive canonical volume. Similar to 
Theorem~\ref{Canonical vol}, we obtain the following:
\begin{Corollary}(Corollary \ref{general type})\label{general type int}
Let $F/\mathbb Q$ be a totally real field of degree $n$. Let
$X_1(\eta)$ be the cover of the Hilbert modular variety parametrizing
abelian varieties with real multiplication by $\mathcal O_F$,
together with a torsion point having annihilator $\eta$.

Every subvariety of $X_1(\eta)$ is of general type provided that
\(
|\Nm(\eta)|>(2\pi)^{2n}.
\)
\end{Corollary}

The variety \(X_1(\eta)\) is defined over \(\Q\), and, given Corollary~\ref{general type int}, the Bombieri--Lang conjecture predicts that if $|\Nm(\eta)|> (2\pi)^{2n},$ then for every number field $K,$
\[
|X_1(\eta)(K)|<\infty.
\]

Although the finiteness predicted by the Bombieri--Lang conjecture remains largely open, combining our effective estimate in Theorem~\ref{log deg} with the recent results of Brunebarbe--Maculan~\cite{brunebarbe2022counting} gives us an unconditional effective bound on the growth rate of \(K\)-rational points in the following sense:

Fix a number field \(K\) and an ideal \(\eta\) relatively prime to the discriminant \(\Delta_F\). Let \(\Xb(\eta)\) be a smooth toroidal compactification of \(X_1(\eta)\), with boundary divisor \(D\), and let
\(X_1(\eta)^*\) be the Baily-Borel compactification of $X_1(\eta).$
Let \(L:=K_{\Xb(\eta)}+D\) be the log-canonical bundle on  \(\Xb(\eta)\) and
let \(\pi:\Xb(\eta)\longrightarrow X_1(\eta)^*\)
be the Baily--Borel contraction. Choose \(b\geq1\) and a very ample line bundle \(A\) on \(X_1(\eta)^*\), defined over \(\mathbb Q\), such that
\(\pi^*A_{\mathbb C}\simeq bL.\) A fixed projective embedding defined by \(A\) determines a multiplicative height \(H_A\) on \(X_1(\eta)^*(\overline{\mathbb Q})\). We define
\[H_L(x):=H_A(x)^{1/b},
\qquad
x\in X_1(\eta)(\overline{\mathbb Q}),
\]
and keep this height fixed. See Section~\ref{sparsity} for further details.

\begin{theorem}(Theorem \ref{log grow}) \label{main sparsity}
Assume that \(|\Nm(\eta)|\geq5^n\) and that \(\eta\) is relatively prime to the discriminant \(\Delta_F\).
For every number field \(K\), set
\(
N_{\eta,K}(B)
=\#\{x\in X_1(\eta)(K):H_L(x)\leq B\}.
\)
Then
\[
\limsup_{B\to\infty}
\frac{\log\max\left\{1,N_{\eta,K}(B)\right\}}
{\log B}
\leq \delta_{\eta,K,n},
\]
where
\[
\delta_{\eta,K,n}
=\frac{\sqrt5[K:\mathbb Q]n(n+3)}
{|\Nm(\eta)|^{1/(2n)}}.
\]
\end{theorem}

\begin{Remark*}[\textbf{Sparsity of rational points}]
In Theorem \ref{main sparsity}, as \(|\Nm(\eta)|\to\infty\) through ideals prime to \(\Delta_F\), the exponent \(\delta\) tends to zero, uniformly when \([K:\mathbb Q]\) and $n$ are bounded. Therefore, as we go higher in the tower \(X_1(\eta)\), its \(K\)-rational points become sparser.
\end{Remark*}

\subsection*{Previous results and comparison} 
\begin{enumerate}[leftmargin=*]
\item \textbf{Geometric torsion over curve bases.}
Cadoret--Tamagawa
\cite{cadoret2011weak,cadoret2012uniform} proved a fixed-family version
of the geometric torsion conjecture. Their approach uses
$\ell$-adic monodromy and uniform open-image methods. Bakker and
Tsimerman \cite{bakker2018geometric} proved the geometric torsion
conjecture for abelian varieties with real multiplication, with bounds in
terms of the genus of the base curve and, after fixing the field of
multiplication, its gonality. Their proof further develops the technique initiated by Hwang and To in \cite{hwang2002volumes,Hwangpolydisk,HwangSeshadri}. More recently, Looper and Yap \cite{looperYap2026} proved the
geometric torsion conjecture in characteristic zero for arbitrary
traceless abelian varieties, with a bound depending on the gonality of the base curve. Their approach
follows the strategy of Hindry--Silverman. Gao and Gu
\cite{gaoGu2026} subsequently gave a different proof, inspired by
Vojta's method, and obtained explicit genus-dependent bounds.

Our approach is closer in spirit to that of Bakker--Tsimerman. We
establish quantitative hyperbolicity properties of torsion level covers
of Hilbert modular varieties by obtaining lower bounds for the canonical
volumes of their subvarieties. These hyperbolicity results also show that
the Bombieri--Lang conjecture predicts strong finiteness properties for
rational points on sufficiently deep torsion level covers. Since our
estimates apply to subvarieties of arbitrary dimension, they also yield
quantitative sparsity results for rational points on these covers. Such
sparsity results seem not to be within reach of methods that study only curves
in the moduli space.

\item \textbf{Full level versus torsion level structures.}
Hyperbolicity in towers of full-level congruence covers of
locally symmetric varieties has been studied extensively; see, for
example,
\cite{nadel1989nonexistence,hwangTo2006level,
brunebarbe2020increasing,brunebarbe2020strong,
abramovich2018level,abramovich2017level,wong2025,wongYeung2025caratheodory}, and \cite{kadets2022level} in positive
characteristic. Brunebarbe's results in particular give canonical volume growth for
subvarieties in the towers of covering with full level structure. A crucial mechanism in the full level structure is that the covering maps ramify
to increasingly high order along every boundary component. This
ramification allows one to use the positivity gained along the tower to
deduce hyperbolicity.

The torsion level covers considered here behave differently.
The space \(X_1(\eta)\) parametrizes a single torsion point rather than a
full level structure, and the map
\(
X_1(\eta)\to X(1)
\)
does not have uniformly increasing ramification at all cusps. This
phenomenon already occurs for modular curves: the map
\(X_1(p)\to X(1)\) has \((p-1)/2\) cusps that are unramified; see
\cite[p.~26]{shimura1971introduction} and
\cite{ogg1972rational}. Thus, the standard argument for the full level
covers cannot be applied directly to the torsion level covers.

 \item\textbf{Rational points versus integral points.} 
It is observed in \cite{ellenberg2023sparsity,brunebarbe2022counting} that if one has control over the degree of all subvarieties, the bound on the growth rate of rational points improves in the strategy of Bombieri-Pila \cite{bombieri1989number} and Heath-Brown \cite{heath2002density}. However, to get the lower bound on the degree of subvarieties, they passed to \'etale covers and this restricts them to obtain a result only on the integral points, rather than rational points. The point is that for rational points on a nonproper variety, the associated
fiber torsors need not be unramified outside a common finite set of places,
so one cannot in general reduce to finitely many twists. For integral points,
by contrast, such uniform ramification control is available
(see Remark~\ref{Chevalley--Weil} for a precise discussion).

Our intrinsic approach has the advantage that it does not require passage to a cover to raise the degree of subvarieties and hence we can get the bound on the growth rate of rational points.
    In general, bounding the growth rate of rational points on a quasi-projective variety is more difficult than bounding the growth rate of integral points. For example, on $X=\mathbb{P}_{K}^1\setminus \{0,1,\infty\},$ there are infinitely many $K$-rational points, however, there are only finitely many integral points on $X$ because of the famous theorem of Siegel and the growth rate of integral points in this case is zero (see \cite[Remark 3.3]{brunebarbe2022counting}).

\item \textbf{Kodaira dimension of Hilbert modular varieties.}
Tsuyumine~\cite{tsuyumine1985kodaira} proved that Hilbert modular varieties
(without level structure) are of general type when \(n>6\), by constructing
pluricanonical forms from Hilbert modular forms and controlling their vanishing
along the cusps and exceptional divisors. His estimates were later used in
\cite[Theorem~D]{cadorel2019hyperbolicity} to show, up to finitely many exceptions, all subvarieties are of general type outside a
proper exceptional locus. This approach, however, does not appear to identify
that exceptional locus.

Bakker--Tsimerman~\cite[Corollary~G]{bakker2018geometric} recovered
Tsuyumine's result by a different method. Building on their work, we show
that, after passing to an effectively computable sufficiently deep level,
every subvariety not contained in the boundary divisor \(D\) is of general
type. Thus, unlike the exceptional locus in \cite[Theorem~D]{cadorel2019hyperbolicity}, the exceptional locus in our result is explicitly controlled: it is contained in the boundary divisor \(D\).
\end{enumerate}
\subsection*{Strategy of proofs} 
Theorem \ref{Canonical vol}:
 The core idea is to identify a region in the locally symmetric space that
is both ``algebraically thick'' and ``geometrically thick'': every
subvariety must intersect this region, and its injectivity radius has a
lower bound tending to infinity on the tower of torsion covers. Once such a region is found, we apply
a theorem of Hwang--To to obtain lower bounds for the hyperbolic volumes
of subvarieties. The passage from hyperbolic volume to canonical volume
is provided by the higher-dimensional Arakelov inequality recently
proved by Brunebarbe \cite[Theorem~1.6]{brunebarbe2020increasing}. To
apply this inequality, we must subtract a fixed multiple of the boundary
divisor from the log-canonical bundle. This is the only step for which we
use the earlier result of Bakker--Tsimerman.

In this paper,  the thick part is the complement of the canonical
horoball neighborhoods of the cusps. In Theorem~\ref{alg thick}, we prove,
using the monodromy of subvarieties, that every subvariety of
\(X_1(\eta)\) meets this thick part. In Proposition~\ref{inj thick}, we
show that, provided \(X_1(\eta)\) has no elliptic points, the injectivity
radius of its thick part has a lower bound depending only on
\([F:\mathbb Q]\) and \(\lvert\operatorname{Nm}(\eta)\rvert\). This
estimate is uniform in \(F\); it depends on the field of multiplication
only through its degree \(n\), and is independent of the choice of
toroidal compactification.

Compared with the full-level results
mentioned above, our ``increasing hyperbolicity'' results arise from
the growth of the injectivity radius on the thick part, rather than
from ramification at the cusps. 
Compared with the ball quotient case in
\cite{memariansorkhabi2025volumes}, the thick region for Hilbert modular
varieties is defined differently. The key properties required of the
thick region are the same, but its definition and the arguments
used to establish these properties are somewhat different. 

Theorem \ref{main sparsity}: We combine the estimate for log-canonical
degrees from Theorem~\ref{log deg} with Brunebarbe--Maculan's recent
refinement of the determinant method in \cite[Theorem 4.4]{brunebarbe2022counting}.

\section*{Acknowledgments}
I would like to thank my advisor Jacob Tsimerman
for many enlightening discussions and his constant support. Also, I would like to thank Daniel Litt who asked me about the uniformity of Theorem \ref{Canonical vol} which led to a correction of the proof of this theorem.  

The author used GPT-5.6 Sol during the preparation of this manuscript to assist with language editing, polishing and organization.

\numberwithin{theorem}{section}
\section{Hilbert modular varieties}\label{Back}

We recall the setup for Hilbert modular varieties, following \cite[\S\S1 and~3]{bakker2018geometric}. In particular, we adopt the conventions of \cite[\S1]{bakker2018geometric} regarding cusps, their neighborhoods, and toroidal compactifications.  We use the notation and definitions introduced here throughout the paper without further reference.
All estimates below are uniform in the choice of the projective $\mathcal O_F$-module $M$.

\subsection{Hilbert modular varieties}\label{hilbert}
Let $\mathbb{H}=\{z\in\C\mid \operatorname{Im} z>0\}$ denote the upper half-plane, and let
\[\mathbb H^n:=\underbrace{\mathbb H\times\cdots\times\mathbb H}_{n \ \text{times}}\]
be the product of \(n\) copies of the upper half plane. The group \(G=\SL_2(\mathbb R)^n\) acts holomorphically on
\(\mathbb H^n\) by M\"obius transformations component-wise.

Let $F$ be a totally real field of degree $n$, with ring of integers $\mathcal O_F$ and real embeddings $\sigma_j:F\hookrightarrow\mathbb R$, $j=1,\ldots,n$; write \(\mathfrak d_F\) for its different and \(\Delta_F=|\Nm(\mathfrak d_F)|\) for its absolute discriminant. Fix a projective $\mathcal O_F$-module $M$ of rank $2$. Since $\mathcal O_F$ is a Dedekind domain, there is a nonzero fractional ideal $\mathfrak a$ such that $M\cong\mathcal O_F\oplus\mathfrak a$. Replacing $\mathfrak a$ by an integral ideal in the same ideal class, we may assume that $\mathfrak a\subset\mathcal O_F.$ The Hilbert modular group associated with \(M\) is
\[
\Gamma(1):=\operatorname{SL}(M)
=
\left\{
g\in\operatorname{Aut}_{\mathcal O_F}(M)
\;\middle|\;
\operatorname{det}_{F}(g)=1
\right\},
\]
where \(\det_F(g)\) denotes the determinant of the induced
\(F\)-linear automorphism of \(M\otimes_{\mathcal O_F}F\).
Relative to
\(
M=\mathcal O_F\oplus\mathfrak a,
\)
its elements have the form
\[
g=
\begin{pmatrix}
\alpha&\beta\\
\gamma&\delta
\end{pmatrix},
\qquad
\alpha,\delta\in\mathcal O_F,\quad
\beta\in\mathfrak a^{-1},\quad
\gamma\in\mathfrak a,\quad
\alpha\delta-\beta\gamma=1.
\]
For each $j$, the identification
\(
M_j:=M\otimes_{\mathcal O_F,\sigma_j}\mathbb R\cong\mathbb R^2
\)
sends $g$ to
\[
g_j=
\begin{pmatrix}
\sigma_j(\alpha)&\sigma_j(\beta)\\
\sigma_j(\gamma)&\sigma_j(\delta)
\end{pmatrix}
\in\operatorname{SL}_2(\mathbb R).
\]
The resulting homomorphism
\[
\Gamma(1)\longrightarrow\operatorname{SL}_2(\mathbb R)^n,
\qquad
g\longmapsto(g_1,\ldots,g_n),
\]
is injective and has arithmetic image. We henceforth identify $\Gamma(1)$ with this image. It acts on $\Hu^n$ by
\[
g\cdot(z_1,\ldots,z_n)
=
\left(
\frac{\sigma_j(\alpha)z_j+\sigma_j(\beta)}
     {\sigma_j(\gamma)z_j+\sigma_j(\delta)}
\right)_{j=1}^n,
\]
and we set
\[
X(1):=\Gamma(1)\backslash\Hu^n.
\]
We regard $X(1)$ with its natural orbifold structure, while its underlying coarse analytic space is the
usual Hilbert modular variety.
More generally, for a finite-index subgroup $\Gamma\subset\Gamma(1)$, we write
\[
X_\Gamma:=\Gamma\backslash\Hu^n.
\]
Since $\Gamma$ is an arithmetic lattice, the quotient $X_\Gamma$ is a
quasi-projective variety. Its cusps are the $\Gamma$-equivalence classes
in $\mathbb P^1(F),$ of which there are only finitely many. By the
Baily--Borel theorem \cite{bb}, $X_\Gamma$ admits a normal projective
compactification
\(
X_\Gamma^*
\)
obtained by adjoining one point for each cusp.
If $\Gamma$ is torsion-free, then $X_\Gamma$ is a smooth complex variety and $X_\Gamma\to X(1)$ is a finite orbifold cover.

Over \(\mathbb C\), the Hilbert modular stack \(X(1)\) associated with
\(\Gamma(1)=\operatorname{SL}(M)\) is naturally identified with the
moduli stack of abelian varieties \(A\) of dimension \(n\), equipped
with real multiplication
\(
\iota:\mathcal O_F\hookrightarrow\operatorname{End}(A),
\)
such that
\[
H_1(A,\mathbb Z)\cong M
\]
as \(\mathcal O_F\)-modules, where the \(\mathcal O_F\)-module structure
on \(H_1(A,\mathbb Z)\) is induced by \(\iota\).

\begin{Remark}\label{polarization-data}
Strictly speaking, the moduli problem fixes a polarized projective \(\mathcal O_F\)-module type \((M,\psi)\) and parametrizes triples \((A,\iota,\lambda)\), where \(\lambda\) is a polarization satisfying \(\lambda\circ\iota(a)=\iota(a)^\vee\circ\lambda\) for every \(a\in\mathcal O_F\), and the displayed marking preserves the polarized type. For a family with real multiplication over an integral base, one may choose \(\lambda\) on the generic fiber and extend it after shrinking the base; its type is then constant \cite[Lemma~2.12 and Theorem~2.16]{goren2002lectures}. Only this fixed choice is used. Following \cite[\S~3]{bakker2018geometric}, we suppress \(\lambda\) and \(\psi\) from the notation; all estimates are uniform in this choice.
\end{Remark}

For a nonzero ideal $\eta\subset\mathcal O_F$, let
\[
\Gamma_1(\eta):=
\left\{
\begin{pmatrix}
\alpha&\beta\\
\gamma&\delta
\end{pmatrix}
\in\Gamma(1)
\;\middle|\;
\alpha\equiv1\pmod{\eta},\quad
\gamma\in\eta\mathfrak a
\right\}.
\]
This is precisely the stabilizer of the class of $(1,0)$ in $M/\eta M$. We denote the corresponding level cover by
\[
X_1(\eta):=\Gamma_1(\eta)\backslash\Hu^n.
\]
For an abelian variety \((A,\iota)\) with real multiplication, set
\[
A[\eta]
:=
\bigcap_{\alpha\in\eta}
\ker\bigl(\iota(\alpha):A\longrightarrow A\bigr).
\]
With the convention of Remark~\ref{polarization-data},
\(X_1(\eta)\) parametrizes tuples \((A,\iota,P)\), where
\(
H_1(A,\mathbb Z)\cong M
\)
as \(\mathcal O_F\)-modules and
\[
P\in A[\eta](\mathbb C),
\qquad
\operatorname{Ann}_{\mathcal O_F}(P)=\eta.
\]
\begin{Lemma}\label{no elliptic}
If \(|\Nm(\eta)|>4^n\), then \(\Gamma_1(\eta)\) acts freely on \(\mathbb H^n\), and every parabolic element of \(\Gamma_1(\eta)\) is unipotent.
\end{Lemma}

\begin{proof}
By \cite[Lemma~3.1]{bakker2018geometric}, \(\Gamma_1(\eta)\) acts freely on \(\mathbb H^n\). Let \(g=\left(\begin{smallmatrix}\alpha&\beta\\ \gamma&\delta\end{smallmatrix}\right)\in\Gamma_1(\eta)\) be parabolic. Since \(\alpha\equiv1\pmod{\eta}\) and \(\beta\gamma\in\eta\), the determinant relation gives \(\delta\equiv1\pmod{\eta}\). Therefore, \(\operatorname{tr}(g)\equiv2\pmod{\eta}\). Since \(g\) is parabolic, its trace is \(2\) or \(-2\). The second case would imply \(4\in\eta\), and hence \(|\Nm(\eta)|\leq4^n\), which is a contradiction. It follows that \(\operatorname{tr}(g)=2\), so \(g\) is unipotent.
\end{proof}

\subsection{Neighborhoods of cusps and a toroidal compactification}\label{Neighborhoods of cusps}

The boundary of each factor $\Hu$ is identified with
\(
\partial\Hu=\mathbb P^1(\mathbb R).
\)
M\"obius transformations in $\operatorname{SL}_2(\mathbb R)$ act on
this boundary as well. Applying this in every factor extends the
component-wise action of \\
\(
\Gamma\subset\operatorname{SL}_2(\mathbb R)^n
\)
from $\Hu^n$ to $\mathbb P^1(\mathbb R)^n$. The real embeddings of $F$
define an embedding
\[
\mathbb P^1(F)\hookrightarrow\mathbb P^1(\mathbb R)^n,
\qquad
[a:b]\longmapsto
\bigl([\sigma_1(a):\sigma_1(b)],\ldots,
      [\sigma_n(a):\sigma_n(b)]\bigr).
\]
For the Hilbert modular groups considered here, the rational boundary
components are precisely the points arising from $\mathbb P^1(F)$ under
this embedding. A cusp of $X_\Gamma$ is a $\Gamma$-equivalence class of
such rational boundary points, and
there are only finitely many such equivalence classes.

Let $c$ be a cusp and choose a representative
$[a:b]\in\mathbb P^1(F)$. Since $\operatorname{SL}_2(F)$ acts
transitively on $\mathbb P^1(F)$, there exists
$g\in\operatorname{SL}_2(F)$ such that
\(
g\cdot[a:b]=[1:0],
\)
which corresponds to \\
\(
\infty:=(\infty,\ldots,\infty)
\in\mathbb P^1(\mathbb R)^n.
\)
Replacing $\Gamma$ by the conjugate lattice $g\Gamma g^{-1}$ gives an
isomorphism
\[
\Gamma\backslash\Hu^n
\xrightarrow{\sim}
(g\Gamma g^{-1})\backslash\Hu^n,
\qquad
[z]\longmapsto[g\cdot z].
\]
After conjugating $\Gamma$, we may
assume that the chosen cusp $c$ is represented by $\infty$. This
choice is unique up to left multiplication by an element of the
stabilizer of $\infty$, namely the upper-triangular subgroup
\begin{equation}\label{the upper-triangular subgroup}
P_\infty(\mathbb R)
=
\left\{
p(a,b):
a\in (\mathbb R^\times)^n,\ b\in \mathbb R^n
\right\}
\subset \operatorname{SL}_2(\mathbb R)^n,
\end{equation}
where, for \(a=(a_1,\dots,a_n)\in(\mathbb R^\times)^n, \) and \(b=(b_1,\dots,b_n)\in\mathbb R^n, \) we set
\[
p(a,b)
=
\left(
\begin{pmatrix}
a_1&b_1\\
0&a_1^{-1}
\end{pmatrix},
\dots,
\begin{pmatrix}
a_n&b_n\\
0&a_n^{-1}
\end{pmatrix}
\right).
\]
Let
\[
N(z):=\prod_{j=1}^n\operatorname{Im}(z_j),\qquad
U(s):=\{z\in\mathbb H^n:N(z)>1/s\}.
\]
We call $U(s)$ the horoball of depth $s$ based at $\infty$. Since
\(
N\bigl(p(a,b)z\bigr)=\left(\prod_{j=1}^n a_j^2\right)N(z),
\)
we have
\(
p(a,b)U(s)=U(s'),\) where \(s'=s/\prod_{j=1}^n a_j^2\).

Let $\Gamma_\infty=\Gamma\cap P_\infty(\mathbb R)$ be the stabilizer of
$\infty$, and let $\Lambda_\infty\triangleleft\Gamma_\infty$ be its
subgroup of unipotent translations, identified with a lattice in
$\mathbb R^n.$ Using
\(
\Lambda_\infty\otimes_{\mathbb Z}\mathbb R\cong\mathbb R^n,
\)
we identify $\Hu^n$ with the positive cone
\[
\mathcal H_\infty
:=
\left\{
\zeta\in\Lambda_\infty\otimes_{\mathbb Z}\mathbb C
\;\middle|\;
(\operatorname{Im}\zeta)_j>0\ \text{for all }j
\right\}.
\]
The lattice $\Lambda_\infty$ acts on $\mathcal H_\infty$ by translations $\Lambda_\infty\otimes 1.$ 
The function $N$ induces the product norm
\[
\operatorname{Nm}_\infty(\lambda)
:=
\prod_{j=1}^n\lambda_j,
\qquad
\lambda=(\lambda_1,\ldots,\lambda_n)\in\Lambda_\infty.
\]
Under conjugation by $p(a,b)$, both $N$ and
$\operatorname{Nm}_\infty$ are multiplied by
$\prod_{j=1}^n a_j^2$. Similar to \cite[\S 1]{bakker2018geometric}, we normalize them so that
\begin{align}\label{lattice normalization}
\min_{0\neq\lambda\in\Lambda_\infty}
\left|\operatorname{Nm}_\infty(\lambda)\right|=1.
\end{align}
This normalization removes the ambiguity coming from the choice of
coordinates at the cusp. More generally, let $c_i$ be a cusp with stabilizer $\Gamma_i$ and
translation lattice $\Lambda_i$. The same normalization gives coordinate
forms
\(
\sigma_i^j:\Lambda_i\longrightarrow\mathbb R
\)
and a  function
\(
N_i(\zeta)
:=
\prod_{j=1}^n
\sigma_i^j\bigl(\operatorname{Im}\zeta\bigr)
\)
on the corresponding positive cone. Although the individual coordinates
$\sigma_i^j$ depend on the choice of coordinates at the cusp, their
normalized product $N_i$ does not. For $s>0$, let
\[U_i(s):=\{\zeta\in\mathcal H_i\mid N_i(\zeta)>1/s\},\] and let
\(\iota_{i,s}:\Gamma_i\backslash U_i(s)\to X_\Gamma\) be the natural map.

\begin{definition}(\cite[Def 1.1]{bakker2018geometric})
The \emph{uniform depth} $d$ of the cusps of $X_\Gamma$ is the largest
$s>0$ such that
\begin{enumerate}
    \item the map $\iota_{i,s}$ is injective for every cusp $c_i$;
    \item the subsets
    \(
    \iota_{i,s}\bigl(\Gamma_i\backslash U_i(s)\bigr)
    \subset X_\Gamma
    \)
    are pairwise disjoint\footnote{Pairwise disjointness is not mentioned in
\cite[Definition~1.1]{bakker2018geometric}, but it is verified separately
for the depths considered in \cite[\S 3]{bakker2018geometric} and used
throughout the paper. We incorporate this condition into the definition,
as in \cite[Definition~3.7]{bakker2018kodaira}.}.
\end{enumerate}
\end{definition}
 For $0<s\leq d$,
we identify $\Gamma_i\backslash U_i(s)$ with its image in $X_\Gamma$ and
write
\[
H_i(s):=\iota_{i,s}\bigl(\Gamma_i\backslash U_i(s)\bigr)\subset X_\Gamma.
\]

Toroidal compactifications of $X_\Gamma$ were constructed in
\cite{Cmpftification}. Their construction requires the choice, at each
cusp $c_i$, of a rational polyhedral fan in the corresponding positive
cone that is admissible with respect to the action of
$\Gamma_i/\Lambda_i$. The resulting compactification depends
on these choices, although any two choices admit a common refinement.
When $\Gamma$ has no elliptic points, we choose a regular projective
admissible collection of fans. The associated toroidal compactification
\(\overline X_\Gamma\) is then a smooth projective variety, and the reduced
boundary divisor
\(D:=\bigl(\overline X_\Gamma\setminus X_\Gamma\bigr)_{\mathrm{red}} \)
is a simple normal crossings divisor; see
\cite[\S 1, \S 4]{bakker2018geometric} for more details. We fix such a choice
throughout.

\begin{definition}
The \emph{canonical horoball} of the cusp $c_i$ is
\[
H_i:=H_i(d)\subset X_\Gamma\subset\overline X_\Gamma,
\]
where $d$ is the uniform depth of the cusps of $\Gamma$. 
\end{definition}
Similar to \cite[\S 4]{memariansorkhabi2025volumes}, we isolate a
compact part of $X_\Gamma$ by removing all its canonical
horoballs:

\begin{definition} \label{core}
We define the \emph{thick part} of $X_\Gamma$ to be
\[
X_{\Gamma,\mathrm{thick}}
:=X_\Gamma\setminus\bigcup_iH_i
=\overline X_\Gamma\setminus
\left(D\cup\bigcup_iH_i\right).
\]
\end{definition}
Following \cite[\S 1, pp. 7--8]{bakker2018geometric}, we recall the local coordinates near the boundary \(D\). For a cusp $c_i$, set
\(
T_i:=\Lambda_i\backslash
\bigl(\Lambda_i\otimes_{\mathbb Z}\mathbb C\bigr).
\)
This is an algebraic torus with character lattice
\(
\Lambda_i^\vee:=\operatorname{Hom}_{\mathbb Z}(\Lambda_i,\mathbb Z).
\)
For $\chi\in\Lambda_i^\vee$, the corresponding character is
\[
q^\chi(\zeta):=
\exp\bigl(2\pi\sqrt{-1}\,\chi(\zeta)\bigr).
\]
Indeed, this function is invariant under $\zeta\mapsto\zeta+\lambda$ for
$\lambda\in\Lambda_i$, since $\chi(\lambda)\in\mathbb Z$. Let $\tau$ be a
maximal cone in the chosen regular fan. Its primitive ray generators form
a basis of $\Lambda_i$; if $\chi_1,\ldots,\chi_n$ is the dual basis, then
the associated toric chart has coordinates
\(
q_k:=q^{\chi_k}
\)
and the toric boundary is given by
\(
q_1\cdots q_n=0.
\)

\subsection{Metric and hyperbolic volume} Endow $\Hu$ with its hyperbolic hermitian metric $ds^2_\Hu$ of constant sectional curvature $-1$; explicitly, the associated K\"ahler form is 
$$\omega_\Hu=\frac{1}{2}\operatorname{Im} ds_\Hu^2=\frac{idz\wedge d\bar z}{2y^2}=i\partial\bar\partial (-2\log y)$$
Likewise endow $\Hu^n$ with the invariant metric
$ds^2_{\Hu^n}=\sum_i \pi_i^* ds^2_{\Hu},$
where $\pi_j:\Hu^n\rightarrow \Hu$ is the $j$th projection, and we denote by $\omega_{\Hu^n}$ the associated K\"ahler form. We use this Kähler metric to define hyperbolic volume. For the injectivity-radius estimates, we use the Kobayashi distance
on $\mathbb H^n:$
\[
d_{\mathbb H^n}(z,w)
:=
\max_{1\leq j\leq n}d_{\mathbb H}(z_j,w_j).
\]

Let $\Gamma\subset G$ be a discrete non-degenerate subgroup of finite covolume, and let
\(X=\Gamma\backslash \Hu^n\)
be endowed with the induced Kähler metric $ds_X^2$ and Kähler form $\omega_X$. Let
$\overline X$ be a smooth toroidal compactification of $X$, with boundary divisor
$D=\overline X\setminus X$. By
\cite[Theorem~3.1, Proposition~3.4(a)]{mumford1977hirzebruch},
the invariant metric on \(\Omega_X^1\) is a good metric on
\(\Omega_{\overline X}^1(\log D)\). Hence its determinant metric
\(h\) is good on
\(\det\Omega_{\overline X}^1(\log D)=K_{\overline X}+D\). Since the
first Chern form of \(h\) is \(\omega_X/(2\pi)\),
\cite[Propositions~1.1--1.2, Theorem~1.4]{mumford1977hirzebruch}
shows that
the trivial extension of \(\omega_X/(2\pi)\) is a closed current, with
no boundary residue, representing
\begin{equation}\label{ChernForm}
c_1(K_{\overline X}+D)
=\frac{1}{2\pi}[\omega_X]\in H^{1,1}(\overline X,\mathbb R).
\end{equation}

 Throughout the paper, all varieties and subvarieties are assumed to be
irreducible and of positive dimension. Let $Z\subset \overline X$ be an $m$-dimensional subvariety not contained
in $D$, and set $Z^\circ:=Z\cap X$. The \emph{hyperbolic volume} of $Z$ is
\[
\operatorname{vol}_{X}(Z)
:=
\frac{1}{m!}
\int_{(Z^\circ)_{\mathrm{reg}}}\omega_X^m,
\]
where $(Z^\circ)_{\mathrm{reg}}$ is the smooth locus of $Z^\circ$. By Mumford's theorem for good metrics
\cite{mumford1977hirzebruch},
applied after taking a log resolution of \((Z,D_{|Z})\), this integral
is finite and
\begin{equation}\label{degree-volume}
(K_{\overline X}+D)^m\cdot Z
=
\frac{1}{(2\pi)^m}
\int_{(Z^\circ)_{\mathrm{reg}}}\omega_X^m
=
\frac{m!}{(2\pi)^m}\operatorname{vol}_{X}(Z).
\end{equation}

We recall the Kobayashi-ball form of the volume estimate of
Hwang--To. Here volume is computed with respect to
$\omega_X$, while the injectivity radius is computed with respect to
the Kobayashi distance defined above. 

For \(X=\Gamma\backslash\mathbb H^n\), with torsion-free lattice $\Gamma$, the \emph{injectivity radius} at \(x\in X\) is
\[
\operatorname{inj}_{X}(x)
:=
\frac12\inf_{\gamma\in\Gamma\setminus\{1\}}
d_{\Hu^n}(z,\gamma z),
\]
where \(z\in\mathbb H^n\) is any lift of \(x\). We recall the special case of result of Hwang and To on bounding the hyperbolic volume
from below:

\begin{theorem}(\cite[Theorem 1.1]{hwang2002volumes}) \label{Hwang--To}Let $Z$ be an $m$-dimensional subvariety of $X$ passing through $x\in X.$ Let $r=\operatorname{inj}_{x}(X)$ be the injectivity radius of $x$ on $X.$ Then, the following inequality holds:
\begin{align} \label{Hwanginequilty} 
\operatorname{vol_{X}(Z)}
&\ge
\frac{(4\pi)^m}{m!}\sinh^{2m}(r/2)\cdot \operatorname{mult}_x(Z),
\end{align} 
where $\operatorname{mult}_x(Z)$ is the multiplicity of $x$ on $Z.$
\end{theorem}

With some modifications, Theorem~\ref{Hwang--To} can also be
stated without assuming that $\Gamma$ is torsion-free; see
\cite[Theorem~4.1]{li2025curves} for the case of two-dimensional ball quotients.

\section{Hodge-Theoretic preliminaries}

We recall the Hodge-theoretic facts needed for the proofs of Theorems~\ref{alg thick} and Proposition~\ref{Canonical vol est}. In particular, we prove Proposition~\ref{Arakelov}, which is used in the proof of Proposition~\ref{Canonical vol est}. Our conventions for the Griffiths line bundle and its extension follow \cite[\S\S2--3]{brunebarbe2020increasing}. 
\subsection{The standard variation of Hodge structure}Let $\Gamma\subset\Gamma(1)$ be a torsion-free finite-index subgroup. We will recall the standard integral variation of Hodge structure on $X_{\Gamma}.$
Set
\(
M_{\mathbb Q}:=M\otimes_{\mathbb Z}\mathbb Q,
\)
\(
V_{\mathbb Z}:=\operatorname{Hom}_{\mathbb Z}(M,\mathbb Z),
\)
and
\(
V_{\mathbb Q}:=V_{\mathbb Z}\otimes_{\mathbb Z}\mathbb Q
=\operatorname{Hom}_{\mathbb Q}(M_{\mathbb Q},\mathbb Q).
\)
Fiberwise, the identification $H_1(A,\mathbb Z)\cong M$ identifies $V_{\mathbb Z}$ with $H^1(A,\mathbb Z)$. The action of $\Gamma$ on $M_{\Q}$ induces the monodromy representation
\[
\rho:\Gamma\longrightarrow\operatorname{GL}(V_{\mathbb Q}),\qquad
\rho(\gamma)(\ell):=\ell\circ\gamma^{-1}.
\]
The diagonal action
\(
\gamma\cdot(z,\ell):=(\gamma z,\rho(\gamma)\ell)
\)
therefore defines an integral local system
\[
\mathbb V_{\mathbb Z}:=\Gamma\backslash
\bigl(\mathbb H^n\times V_{\mathbb Z}\bigr)
\]
on $X_\Gamma$. Since $M$ has rank $2$ over $\mathcal O_F$, the local system $\mathbb V_{\mathbb Z}$ has rank $2n$. We write
\(
\mathbb V_{\mathbb Q}:=\mathbb V_{\mathbb Z}\otimes_{\mathbb Z}\mathbb Q,
\) and 
\(
\mathbb V_{\mathbb C}:=\mathbb V_{\mathbb Z}\otimes_{\mathbb Z}\mathbb C.
\)
The local system \(\mathbb V_{\mathbb C}\) underlies the holomorphic
bundle
\[
\mathcal V:=\mathbb V_{\mathbb C}\otimes_{\mathbb C}
\mathcal O_{X_\Gamma}
\]
endowed with its tautological flat connection \(\nabla\), whose
monodromy is \(\rho\).
The corresponding polarization on the standard cohomological local system induces, with the sign chosen so that it polarizes $H^1$, a $\Gamma$-invariant nondegenerate alternating form $Q_V$ on $V_{\mathbb Q}$. After multiplying $Q_V$ by a positive integer, we may assume that its restriction to $V_{\mathbb Z}$ is integral, and it therefore gives a flat pairing on $\mathbb V_{\mathbb Z}$. The real embeddings of $F$ give a decomposition
\[
V_{\mathbb R}:=V_{\mathbb Z}\otimes_{\mathbb Z}\mathbb R
=\bigoplus_{j=1}^nV_j,\qquad
V_j:=\left(M\otimes_{\mathcal O_F,\sigma_j}\mathbb R\right)^\vee.
\]
Since the polarization is compatible with the $\mathcal O_F$-action, the summands $V_j$ are pairwise $Q_V$-orthogonal. Starting from the basis of $V_j$ dual to the standard basis of
\(
M\otimes_{\mathcal O_F,\sigma_j}\mathbb R\cong\mathbb R^2,
\)
we rescale both vectors by the same positive real number and denote the resulting basis by $(e_j,f_j)$, so that
\(
Q_V(e_j,f_j)=-1.
\)
For $z=(z_1,\ldots,z_n)\in\mathbb H^n$, define
\[
F_z^1:=\bigoplus_{j=1}^n\mathbb C(e_j-z_jf_j)\subset V_{\mathbb C}.
\]
If
\(
\gamma_j=\begin{pmatrix}a_j&b_j\\c_j&d_j\end{pmatrix}
\)
is the image of $\gamma\in\Gamma$ in the $j$-th factor, then
\begin{equation}\label{Automporphic}
\gamma_j^{-T}\begin{pmatrix}1\\-z_j\end{pmatrix}
=(c_jz_j+d_j)\begin{pmatrix}1\\-\gamma_jz_j\end{pmatrix}.
\end{equation}
Hence
\(
\rho(\gamma)F_z^1=F_{\gamma z}^1.
\)
The spaces $F_z^1$ therefore descend to a holomorphic subbundle
\[
F^1\mathcal V\subset\mathcal V.
\]
Together with
\(
F^0\mathcal V=\mathcal V
\)
and
\(
F^2\mathcal V=0,
\)
this defines a weight-one Hodge filtration. Since $z_j\notin\mathbb R$, one has
\(
V_{\mathbb C}=F_z^1\oplus\overline{F_z^1},
\)
and hence
\(
h^{1,0}=h^{0,1}=n.
\)
Moreover,
\(
Q_V(F_z^1,F_z^1)=0,
\)
and for
\(
v=\sum_{j=1}^n\alpha_j(e_j-z_jf_j)\in F_z^1
\)
one has
\(
\sqrt{-1}\,Q_V(v,\overline v)
=2\sum_{j=1}^n\operatorname{Im}(z_j)|\alpha_j|^2>0
\)
whenever $v\neq0$. Hence
\(
(\mathbb V_{\mathbb Z},F^\bullet\mathcal V,Q_V)
\)
is a polarized integral variation of Hodge structure of weight $1$ on $X_\Gamma$ and the complex variation $\mathbb V_{\mathbb C}$ has length $1.$

\subsection{The Hodge bundle} Let $\Gamma:=\Gamma_1(\eta)$ and $X:=X_\Gamma=X_1(\eta)$. We assume that $|\Nm(\eta)|>4^n$ so that $\Gamma$ is torsion free and all parabolic elements are unipotent by Lemma \ref{no elliptic}.  Let $\pi:\mathcal A\to X$ be the universal abelian scheme. Its Hodge bundle is
\begin{equation*}
\mathbb E:=\pi_*\Omega^1_{\mathcal A/X}=F^1\mathcal V.
\end{equation*}
For a point \(x\in X\), corresponding to an abelian variety \(A_x\), the fiber
of \(\mathbb E\) is
\(
\mathbb E_x=H^0(A_x,\Omega^1_{A_x}).
\)
 Note that by the real multiplication by \(\mathcal O_F\), we have an embedding
\(
\iota:\mathcal O_F\hookrightarrow \operatorname{End}_X(\mathcal A).
\)
Therefore, every element \(a\in \mathcal O_F\) gives an endomorphism
\[
\iota(a):\mathcal A\longrightarrow \mathcal A.
\]
The action of \(\mathcal O_F\) on \(\mathcal A\) induces an action of
\(\mathcal O_F\) on the Hodge bundle by pullback. Namely, for each
\(a\in\mathcal O_F\), we obtain an \(\mathcal O_X\)-linear endomorphism
\(
\iota(a)^*:\mathbb E\longrightarrow \mathbb E.
\)
Therefore \(\mathbb E\) is naturally an
\((\mathcal O_F\otimes_{\mathbb Z}\mathcal O_X)\)-module. After base change to
\(\mathbb C\), the algebra \(\mathcal O_F\) splits via the real embeddings
\(\sigma_1,\dots,\sigma_n:F\hookrightarrow \mathbb R\subset\mathbb C\):
\[
\mathcal O_F\otimes_{\mathbb Z}\mathbb C
\cong
\bigoplus_{j=1}^n \mathbb C,
\qquad
a\otimes z\longmapsto
\bigl(z\sigma_1(a),\dots,z\sigma_n(a)\bigr).
\]
Hence, over \(X\), we have
\(
\mathcal O_F\otimes_{\mathbb Z}\mathcal O_{X}
\cong
\bigoplus_{j=1}^n \mathcal O_{X}.
\)
This decomposition gives idempotent projectors
\(
p_1,\dots,p_n,
\)
where \(p_j\) is the projector onto the \(j\)-th factor. These satisfy
\[
p_j^2=p_j,\qquad
p_i p_j=0\quad\text{for }i\neq j,\qquad
p_1+\cdots+p_n=1.
\]
Since these projectors act on \(\mathbb E\), they split the Hodge
bundle into their images:
\begin{equation*}
\mathbb E=\bigoplus_{j=1}^n\mathcal L_j,
\end{equation*}
where
\(
\mathcal L_j:=p_j\mathbb E
\)
(see for example \cite[\S 2.3]{goren2002lectures}). Each $\mathcal L_j$ is the automorphic line bundle
corresponding to the $j$-th embedding. Fiberwise, for a
point corresponding to an abelian variety $A$, one has
\begin{equation*}
(\mathcal L_j)_A
=
\left\{
\omega\in H^0(A,\Omega_A^1)
\mid
a^*\omega=\sigma_j(a)\omega
\text{ for all }a\in \mathcal O_F
\right\}.
\end{equation*}
The pull back of \(\mathcal L_j\)  on $\mathbb H^n$ is generated by the section
\(
s_j(z):=e_j-z_jf_j.
\)
For
\(
\gamma=
\begin{pmatrix}
a&b\\
c&d
\end{pmatrix}\in\Gamma,
\)
we write
\(
j_j(\gamma,z):=\sigma_j(c)z_j+\sigma_j(d).
\)
Then \eqref{Automporphic} gives
\[
\rho(\gamma)s_j(z)
=
j_j(\gamma,z)s_j(\gamma z).
\]
Therefore, $\mathcal L_j$ is the automorphic line bundle of weight one at the embedding $\sigma_j$. 

Let
\(
\lambda:=\det\mathbb E.
\)
In additive notation in Picard group, this becomes
\(
\lambda=\mathcal L_1+\cdots+\mathcal L_n.
\)
Let
\(
j:X\hookrightarrow \overline X
\)
be a smooth toroidal compactification with simple normal-crossings boundary
\(
D=\overline X\setminus X.
\)
A small loop around an irreducible component of \(D\) is represented by an
element of the unipotent translation lattice of the corresponding cusp.
Its action on the standard representation, and hence on
\(\mathbb V_{\mathbb C}\), is therefore unipotent.

The flat bundle \((\mathcal V,\nabla)\) introduced above
admits its canonical
Deligne extension
\(
(\overline{\mathcal V},\overline\nabla)
\)
to \(\overline X\), characterized by the fact that
\[
\overline\nabla:
\overline{\mathcal V}
\longrightarrow
\overline{\mathcal V}\otimes
\Omega^1_{\overline X}(\log D)
\]
has nilpotent residues along the irreducible components of \(D\).
The Hodge filtration extends by setting
\[
\overline F^p
:=
\overline{\mathcal V}\cap j_*(F^p\mathcal V)
\subset j_*\mathcal V.
\]
These are locally free subbundles of \(\overline{\mathcal V}\); the resulting
filtered logarithmic bundle
\(
(\overline{\mathcal V},\overline\nabla,\overline F^\bullet)
\)
is the Deligne--Schmid extension of the variation. Since the local
monodromies are unipotent, the associated parabolic structure has only
weight zero. Thus, the parabolic extension is
the Deligne--Schmid extension endowed with the trivial parabolic structure; see \cite[\S 1.2, \S 2.2]{brunebarbe2020increasing}. Let
\[
\overline{\mathbb E}:=\overline F^1,
\qquad
\overline\lambda:=\det\overline{\mathbb E}.
\]
Every element of \(\mathcal O_F\) acts on \(\mathcal V\) by a flat
endomorphism preserving \(F^1\mathcal V\). By the functoriality of the
canonical Deligne extension, this action extends to
\(\overline{\mathcal V}\) and preserves \(\overline{\mathbb E}\).
The idempotent projectors \(p_1,\ldots,p_n\) therefore extend to
\(\overline{\mathbb E}\), and yield a decomposition
\[
\overline{\mathbb E}
=
\bigoplus_{j=1}^n\overline{\mathcal L}_j,
\qquad
\overline{\mathcal L}_j:=p_j\overline{\mathbb E}.
\]
Here \(\overline{\mathcal L}_j\) is the eigenline on which
\(a\in\mathcal O_F\) acts through \(\sigma_j(a)\), and
\(
\overline{\mathcal L}_j|_X\simeq\mathcal L_j.
\)
In particular,
\[
\overline\lambda
=
\det\overline{\mathbb E}
\simeq
\bigotimes_{j=1}^n\overline{\mathcal L}_j.
\]
The Griffiths line bundle of the variation is
\(
L_{\mathbb V_{\mathbb C}}
:=
\bigotimes_p\det F^p\mathcal V.
\)
In weight one, its only contributions come from
\(F^1\mathcal V=\mathbb E\), and hence
\[
L_{\mathbb V_{\mathbb C}}
\simeq
\det\mathbb E
=
\lambda.
\]
 It follows that the Deligne--Schmid extension of the Griffiths line bundle is
\begin{equation}\label{eq:extended-Griffiths-is-Hodge}
\overline L_{\mathbb V_{\mathbb C}}
:=
\bigotimes_p\det\overline F^p
\simeq
\det\overline F^1
=
\overline\lambda.
\end{equation}
This is the Griffiths parabolic line bundle whose parabolic weights
are all zero in our setting. This line bundle is
functorial under pullback. More precisely, if $\overline Y$ is smooth, $E$ is a simple
normal-crossings divisor, and
$f:\overline Y\to\overline X$ satisfies $f^{-1}(D)\subseteq E$,
then, writing $f^\circ:Y:=\overline Y\setminus E\to X$,
\[
\overline L_{(f^\circ)^*\mathbb V_{\mathbb C}}
\simeq f^*\overline\lambda.
\]
This follows from compatibility of the extended Hodge filtration
with pullback; see
\cite[\S 2.2]{brunebarbe2020increasing}.

Now we identify a multiple of this line bundle with the log-canonical bundle of the toroidal compactification. We use a similar automorphy-factor approach as in \cite[Section II.1]{hulek2002geometry}. By \eqref{Automporphic}, $\mathcal L_j^{\otimes2}$ has automorphy factor $j_j(\gamma,z)^2$ and
\(d(\gamma_jz_j)=j_j(\gamma,z)^{-2}dz_j.\)
Hence the map on $\mathbb H^n$ defined by
\(s_j(z)^{\otimes2}\mapsto dz_j\)
is $\Gamma$-equivariant and descends to an isomorphism
\[
\bigoplus_{j=1}^n\mathcal L_j^{\otimes2}\xrightarrow{\sim}\Omega_X^1.
\]

We extend this isomorphism across the boundary. Fix a cusp $c_i$ and let $\zeta=(\zeta_1,\ldots,\zeta_n)$ be the coordinate around the cusp, so that the translation lattice $\Lambda_i$ acts by $\zeta\mapsto\zeta+\tau$ for $\tau\in\Lambda_i$. We use the same notation $(e_j,f_j)$ for the corresponding conjugated flat basis; then \(s_j(\zeta)=e_j-\zeta_jf_j.
\)
For $\tau\in\Lambda_i$, one has
\[\rho(\tau)e_j=e_j-\sigma_j(\tau)f_j,\qquad \rho(\tau)f_j=f_j,
\]
and hence
\(s_j(\zeta+\tau)=\rho(\tau)s_j(\zeta).\) Let $\chi_1,\ldots,\chi_n$ be a basis of $\Lambda_i^\vee$ and let
\(
q_k:=q^{\chi_k},\ k=1,\ldots,n,
\)
be the toroidal coordinates introduced in Subsection~\ref{Neighborhoods of cusps}.
Writing
\(\chi_k(\zeta)=\sum_ja_{kj}\zeta_j\)
and $A=(a_{kj})$, we have
\[
\begin{pmatrix}
dq_1/q_1\\
\vdots\\
dq_n/q_n
\end{pmatrix}
=
2\pi\sqrt{-1}\,A
\begin{pmatrix}
d\zeta_1\\
\vdots\\
d\zeta_n
\end{pmatrix}.
\]
Since $\chi_1,\ldots,\chi_n$ form a basis of $\Lambda_i^\vee$, the matrix $A$ is invertible. Moreover, since $e_j$ and $f_j$ are flat,
\[\nabla s_j=-d\zeta_j\otimes f_j
=-\frac{1}{2\pi\sqrt{-1}}\sum_k(A^{-1})_{jk}\frac{dq_k}{q_k}\otimes f_j,
\qquad
\nabla f_j=0.
\]
Therefore the bundle generated by $s_j$ and $f_j$ has a logarithmic connection, whose residue along $q_k=0$ sends
\[s_j\longmapsto-\frac{1}{2\pi\sqrt{-1}}(A^{-1})_{jk}f_j,
\qquad
f_j\longmapsto0.
\]
Its residues are therefore nilpotent, so this bundle is the canonical Deligne extension and $s_j$ is a local frame of $\overline{\mathcal L}_j$. Since $A$ is invertible, the forms $d\zeta_j$ and $dq_k/q_k$ are related by an invertible constant change of basis. Consequently, the map $s_j(\zeta)^{\otimes2}\mapsto d\zeta_j$ extends without zeros or poles to an isomorphism  on the toric cover. This extension is $\Gamma_i/\Lambda_i$-equivariant and hence descends to a neighborhood of the boundary in $\overline X$ giving
\[
\Omega_{\overline X}^1(\log D)
\simeq\bigoplus_{j=1}^n\overline{\mathcal L}_j^{\otimes2};
\]
see \cite[Theorem~5.2(iii)]{DimitrovTilouine2004}. Taking determinant gives
\(
K_{\overline X}(D)
\simeq\bigotimes_{j=1}^n\overline{\mathcal L}_j^{\otimes2}
\simeq\overline\lambda^{\otimes2}
\simeq\overline L_{\mathbb V_{\mathbb C}}^{\otimes2}.
\)
In $\operatorname{Pic}(\overline X)$, this means
\begin{equation}\label{log-Griffiths}
K_{\overline X}+D
=2\overline L_{\mathbb V_{\mathbb C}}.
\end{equation}

We now apply Brunebarbe's higher-dimensional Arakelov inequality. Since the standard weight-one variation has rank $2n$ and length $1$, the following is a direct consequence of \cite[Theorems~1.6 and~2.3]{brunebarbe2020increasing}:

\begin{Proposition}\label{Arakelov}
Let $Z\subset\overline X$ be a positive-dimensional irreducible subvariety not contained in $D$, and let $\nu:\widetilde Z\to Z$ be a log resolution of $(Z,D|_Z)$. Set
\(
\widetilde D:=\bigl(\nu^*(D|_Z)\bigr)_{\mathrm{red}}.
\)
Then
\[
K_{\widetilde Z}+\widetilde D-\frac{1}{n}\nu^*
\left(\left.(K_{\overline X}+D)\right|_Z\right)
\]
is pseudoeffective.
\end{Proposition}

\begin{proof}
Let $f:\widetilde Z\to\overline X$ be the morphism induced by $\nu$, and let $\mathbb W$ be the pullback of $\mathbb V_{\mathbb C}$ to $\widetilde Z\setminus\widetilde D$. The standard period map on $\mathbb H^n$ is immersive, and $f$ is generically immersive since $\nu$ is birational and $Z$ is not contained in $D$. Hence the period map of $\mathbb W$ is generically immersive. Pullback preserves rank and length, so $\mathbb W$ has rank $2n$ and length $1$. Therefore \cite[Theorems~1.6 and~2.3]{brunebarbe2020increasing} implies that
\[
K_{\widetilde Z}+\widetilde D-\frac{2}{n}\overline L_{\mathbb W}
\]
is pseudoeffective. By functoriality of the Griffiths parabolic line bundle and \eqref{log-Griffiths},
\[
\overline L_{\mathbb W}
\simeq f^*\overline L_{\mathbb V_{\mathbb C}}
\simeq\frac{1}{2}f^*(K_{\overline X}+D)
=\frac{1}{2}\nu^*
\left(\left.(K_{\overline X}+D)\right|_Z\right);
\]
see \cite[\S\S 2.2 and 2.4]{brunebarbe2020increasing}. Substitution proves the proposition.
\end{proof}

\section{Hyperbolic volume estimates}

Let \(\Gamma\subseteq\Gamma(1)\) be a torsion-free subgroup of finite index, and let \(\overline X\) be a smooth projective toroidal compactification of \(X:=X_\Gamma\), with boundary divisor \(D:=\overline X\setminus X_\Gamma\). Throughout the paper, a (sub)variety is assumed to be positive-dimensional and irreducible. We prove in Theorem~\ref{alg thick} that every subvariety of \(\overline X\) not contained in \(D\) meets the thick part of \(X\). We then establish in Proposition~\ref{inj thick} a uniform lower bound for the injectivity radius on the thick part of $X_1(\eta)$. Therefore, every such subvariety contains a point with large injectivity radius and then we conclude the hyperbolic volume estimates we need in Theorem \ref{log deg}.

\begin{theorem}\label{alg thick} Every subvariety
$Z\subseteq\overline X$ either intersects $X_{\mathrm{thick}}$ or is fully contained in the boundary $D$. 

\end{theorem}

We need the following two basic lemmas to prove Theorem~\ref{alg thick}. The first shows that cusp stabilizers of Hilbert modular varieties are solvable:
\begin{Lemma}\label{Metabelain} Let $P_\infty(\mathbb R)$ be the upper-triangular group defined in \eqref{the upper-triangular subgroup}.
Let $\Gamma_\infty\subseteq P_\infty(\mathbb R)$ be a cusp
stabilizer of $\Gamma$, and let
$\rho:\Gamma_\infty\to\operatorname{GL}(V_{\mathbb Q})$
be a representation on a finite-dimensional $\mathbb Q$-vector space.
Then the algebraic group
\[
\overline{\rho(\Gamma_\infty)}^{\,\mathrm{Zar}}_{\mathbb Q}
\subseteq\operatorname{GL}(V_{\mathbb Q})
\]
is solvable.
\end{Lemma}

\begin{proof}
The surjective homomorphism
\[
P_\infty(\mathbb R)\longrightarrow(\mathbb R^\times)^n,\qquad
p(a,b)\longmapsto a,
\]
has abelian kernel
\(
U_\infty(\mathbb R)
=
\{p(1,b):b\in\mathbb R^n\}
\simeq(\mathbb R^n,+).
\)
Since its image is also abelian,
\(
[P_\infty(\mathbb R),P_\infty(\mathbb R)]
\subseteq U_\infty(\mathbb R),
\)
so $P_\infty(\mathbb R)$ is metabelian. Therefore
$\Gamma_\infty$ and $\rho(\Gamma_\infty)$ are metabelian,
and in particular solvable. Since the Zariski closure of a
solvable subgroup is solvable, $\overline{\rho(\Gamma_\infty)}^{\,\mathrm{Zar}}_{\mathbb Q}$ is solvable.
\end{proof}

We recall the following standard fact about linear algebraic groups:
\begin{Lemma}\label{semisimple-solvable}
Let $ H$ be a connected linear algebraic group. If $H$ is both
solvable and semisimple, then $ H$ is trivial.
\end{Lemma}

\begin{proof}
Let $R(H)$ denote the radical of $H$, i.e. the maximal connected
solvable normal algebraic subgroup of $ H$. Since $H$ is connected
and solvable, we have
\(
R(H)= H.
\)
On the other hand, since $ H$ is semisimple, its radical is trivial:
\(
R( H)=\{1\}.
\)
Therefore $H=\{1\}$.
\end{proof}

Now we can prove Theorem \ref{alg thick}:

\begin{proof}[Proof of Theorem \ref{alg thick}]

Suppose that a subvariety $Z\subset\overline X$ not contained in $D$
does not meet $X_{\mathrm{thick}}$. Then the dense open subset
\(
Z^\circ:=Z\cap X
\)
is contained in the union of the canonical horoballs. Since $Z^\circ$
is connected, and these horoballs are pairwise
disjoint, it is contained in a single canonical horoball $H_i$ around
a cusp $c_i$. After conjugating $\Gamma$, we may assume that $c_i$ is
represented by $\infty$.

Let $\nu:\widetilde Z\to Z$ be a projective log resolution of
$(Z, D_{|Z})$ such that
\(
\widetilde{D}:=\bigl(\nu^{-1}(Z\cap D)\bigr)_{\mathrm{red}}
\)
has simple normal crossings, and set
\(
\widetilde Z^\circ:=\widetilde Z\setminus \widetilde{D}.
\)
The induced map
\(
f:\widetilde Z^\circ\to X
\)
factors through $H_i$. Since $U_i$ is contractible, $\Gamma_\infty$ acts freely on $U_i,$
and $H_i=\Gamma_\infty\backslash U_i$, we
have
\(
\pi_1(H_i)\simeq\Gamma_\infty.
\)
Therefore, after choosing compatible basepoints, \(f\) induces the
homomorphism
\[
u:\pi_1(\widetilde Z^\circ)\longrightarrow
\Gamma_\infty\subseteq\Gamma,\qquad
\Gamma_\infty\subseteq P_\infty(\mathbb R).
\]
The monodromy of the pullback of the standard variation is
\[
\rho\circ u:
\pi_1(\widetilde Z^\circ)
\longrightarrow\operatorname{GL}(V_{\mathbb Q}).
\]
 Consider the connected algebraic monodromy group
\[
\mathbf H_Z
:=
\left(
\overline{
(\rho\circ u)\bigl(\pi_1(\widetilde Z^\circ)\bigr)
}^{\,\mathrm{Zar}}_{\mathbb Q}
\right)^\circ
\subseteq\operatorname{GL}(V_{\mathbb Q}).
\]
Since
\(
(\rho\circ u)\bigl(\pi_1(\widetilde Z^\circ)\bigr)
\subseteq \rho(\Gamma_\infty),
\)
we have
\(
\mathbf H_Z
\subseteq
\overline{\rho(\Gamma_\infty)}^{\,\mathrm{Zar}}_{\mathbb Q}.
\)
By Lemma~\ref{Metabelain}, $\overline{\rho(\Gamma_\infty)}^{\,\mathrm{Zar}}_{\mathbb Q}$ is solvable,
so $\mathbf H_Z$ is solvable.

On the
other hand, the pullback of the standard variation to
\(\widetilde Z^\circ\) is a polarized pure \(\mathbb Q\)-variation of Hodge
structure. Therefore, André--Deligne semisimplicity
\cite[Corollary~1]{andre1992mumford} implies that the group \(\mathbf H_Z\) is semisimple. It follows
from Lemma~\ref{semisimple-solvable} that
\(\mathbf H_Z=\{1\}\). Therefore,
\((\rho\circ u)\bigl(\pi_1(\widetilde Z^\circ)\bigr)\) is finite.
Since \(\rho\) is faithful and \(\Gamma\) is torsion-free, we obtain
\[
u\bigl(\pi_1(\widetilde Z^\circ)\bigr)=\{1\}.
\]
Consequently, $f$ lifts to a holomorphic map
\(
\widetilde f:\widetilde Z^\circ\longrightarrow\mathbb H^n.
\)
Identifying $\mathbb H^n$ with the bounded polydisc $\Delta^n$, the
coordinate functions of $\widetilde f$ are bounded and therefore extend
holomorphically across $\widetilde{D}$. The resulting holomorphic functions on the
compact connected variety $\widetilde Z$ are constant. Hence
$\widetilde f$, and therefore $f$, is constant. This contradicts
$\dim Z>0$.
\end{proof}

We denote by \(X_{1,\mathrm{thick}}(\eta)\) the thick part of \(X_1(\eta)\) defined in Definition~\ref{core}.

\begin{Proposition}\label{inj thick}
For every \(x\in X_{1,\mathrm{thick}}(\eta)\), the following hold:
\[
\begin{array}{ll}
\mathrm{(i)}&\text{If \( |\Nm(\eta)|> 4^n\), then }
\displaystyle\operatorname{inj}_{X_1(\eta)}(x)\ge\frac{1}{2n}\log|\Nm(\eta)|-\frac12.\\[4pt]
\mathrm{(ii)}&\text{If \( |\Nm(\eta)|\ge5^n\), then }
\displaystyle\operatorname{inj}_{X_1(\eta)}(x)
\ge
\operatorname{arsinh}\left(\frac12|\Nm(\eta)|^{1/(2n)}\right)
\ge
\frac{1}{2n}\log|\Nm(\eta)|.
\end{array}
\]
\end{Proposition}
\begin{proof}
Under either hypothesis, Lemma \ref{no elliptic} implies that $\Gamma_1(\eta)$ is torsion-free and every parabolic element is unipotent.

Let \(z\in\Hu^n\) be a lift of \(x\), and let
\(1\ne\gamma\in\Gamma_1(\eta)\). If \(\gamma\) is semisimple, then
\cite[Corollary~3.10]{bakker2018geometric} gives
\[
\frac12d_{\Hu^n}(z,\gamma z)
\ge \frac1n\log|\Nm(\eta)|-1,
\]
and the desired bound follows for this case. Suppose now that \(\gamma\) is parabolic and stabilizes the cusp \(c_i\).
In coordinates centered at \(c_i\), write
\(
\gamma z=z+\lambda
\)
with \(0\ne\lambda\in\Lambda_i\). Setting
\(y_j=\operatorname{Im}(z_j)\), we have
\[
d_{\Hu}(z_j,z_j+\lambda_j)
=
2\sinh^{-1}\left(\frac{|\lambda_j|}{2y_j}\right)
\ge
2\log\left(\frac{|\lambda_j|}{y_j}\right).
\]
Since the Kobayashi distance on \(\Hu^n\) is the maximum of the
coordinatewise distances,
\begin{align*}
d_{\Hu^n}(z,z+\lambda)
&\ge
2\max_j\log\left(\frac{|\lambda_j|}{y_j}\right)\\
&\ge
\frac{2}{n}\sum_{j=1}^n
\log\left(\frac{|\lambda_j|}{y_j}\right)
=
\frac{2}{n}
\log\left(\frac{|\Nm_i(\lambda)|}{N_i(z)}\right).
\end{align*}
By the normalization \eqref{lattice normalization},
\(
|\Nm_i(\lambda)|\ge1.
\)
Moreover, since \(x\in X_{1,\mathrm{thick}}(\eta)\),
\(
N_i(z)^{-1}\ge d,
\)
where \(d\) is the uniform depth of the cusps. Hence
\cite[Corollary~3.5]{bakker2018geometric} yields
\[
d_{\Hu^n}(z,\gamma z)
\ge\frac{2}{n}\log d
\ge\frac{1}{n}\log|\Nm(\eta)|,
\]
and the claimed bound for (i) follows.

Suppose now that \(|\Nm(\eta)|\ge5^n\). For semisimple \(\gamma\), the nonzero element \(\operatorname{tr}(\gamma)-2\in\eta\) has an embedding \(\sigma_j\) satisfying
\(
\left|\sigma_j\bigl(\operatorname{tr}(\gamma)\bigr)-2\right|\ge|\Nm(\eta)|^{1/n}.
\)
Therefore
\[
\frac12d_{\Hu^n}(z,\gamma z)
\ge\operatorname{arcosh}\left(\frac12|\Nm(\eta)|^{1/n}-1\right)
\ge\sinh^{-1}\left(\frac12|\Nm(\eta)|^{1/(2n)}\right),
\]
where the last inequality follows from \(|\Nm(\eta)|^{1/n}\ge5\). If $\gamma$ is parabolic, set $u_j=|\lambda_j|/y_j$. Then
\[
\prod_{j=1}^nu_j=\frac{|\Nm_i(\lambda)|}{N_i(z)}\ge d\ge|\Nm(\eta)|^{1/2},
\]
and hence
\[
\frac12d_{\Hu^n}(z,\gamma z)=\max_j\operatorname{arsinh}\left(\frac{u_j}{2}\right)
\ge\operatorname{arsinh}\left(\frac12\left(\prod_{j=1}^nu_j\right)^{1/n}\right)
\ge\operatorname{arsinh}\left(\frac12|\Nm(\eta)|^{1/(2n)}\right).
\]
The remaining inequality follows from $\operatorname{arsinh}(t/2)\ge\log t$ for $t>0$. This proves (ii).
\end{proof}

Now we can conclude a uniform bound on the degrees of all subvarieties with respect to the log-canonical bundle:  
\begin{theorem}\label{log deg} Suppose that $|\Nm(\eta)|>4^n$. Let $\Xb(\eta)$ be a smooth toroidal compactification of $X_1(\eta)$ with boundary divisor $D$ and let $Z$ be an irreducible subvariety of $\Xb(\eta)$ which is not contained in $D.$ Then,
\begin{align*}
    (K_{\Xb(\eta)}+D)^m \cdot Z
    \ge
    2^m\sinh^{2m}\bi( \frac{1}{4n}\log |\Nm(\eta)|-\frac{1}{4}\bi),
\end{align*}
where $m$ is the dimension of $Z.$ Moreover, if $|\Nm(\eta)|\ge5^n$, then
\[
(K_{\Xb(\eta)}+D)^m\cdot Z
\ge2^m\sinh^{2m}\left(\frac12\operatorname{arsinh}\left(\frac12|\Nm(\eta)|^{1/(2n)}\right)\right).
\]
\end{theorem}
\begin{proof}
By Theorem \ref{alg thick}, the variety $Z$ intersects $X_{1,\mathrm{thick}}(\eta)$ and therefore by Proposition \ref{inj thick} the subvariety $Z$ contains a point with injectivity radius at least $\frac{1}{2n}\log |\Nm(\eta)|-\frac12$ on $X_1(\eta).$ Hence, Hwang-To's inequality \eqref{Hwanginequilty}, together with \eqref{degree-volume}, provides us a uniform bound for the intersection of $Z$ with the log-canonical bundle:
\begin{align*}
    (K_{\Xb(\eta)}+D)^m \cdot Z =  \frac{m!}{(2\pi)^m}\operatorname{vol}_{X_1(\eta)}(Z)
    \ge 
    2^m\sinh^{2m}( \frac{1}{4n}\log |\Nm(\eta)|-\frac{1}{4}).
\end{align*}
If $|\Nm(\eta)|\ge5^n$, the same argument using Proposition~\ref{inj thick}(ii) gives the second part. \end{proof}
\begin{Remark} Note that the right-hand side of the inequalities in Theorem \ref{log deg} depends only on the norm of the ideal $\eta$ and the degree of the field of multiplication over $\Q$ and therefore it is independent of the choice of field or the compactification.   
\end{Remark}

\section{Canonical Volume}

The volume of a line bundle $L$ on an $m$-dimensional projective variety $V$ is defined to be the non-negative real number 
$$\operatorname{vol}_{V}(L):=\limsup\limits_{b\rightarrow \infty}\dfrac{h^0(V, bL)}{b^m/m!},$$
which measures the positivity of $L$ from the point of view of birational geometry. If $L$ is a nef line bundle on $V$, then $\operatorname{vol}_{V}(L)= L^{m}.$ Also, $L$ is big if and only if $\operatorname{vol}_{V}(L)>0.$

The canonical volume of a quasi-projective $Z$ is the volume of the canonical bundle of any smooth projective variety birational to $Z$. More explicitly, if $\overline{Z}$ is a smooth projective variety birational to $Z$, then
\[
\widetilde{\operatorname{vol}}_{Z}
:=
\operatorname{vol}_{\overline{Z}}(K_{\overline{Z}}).
\]
is the canonical volume of $Z$. This quantity is independent of the choice of the smooth projective model $\overline{Z}$. In particular, when $Z$ is an integral curve, then
\(
\widetilde{\operatorname{vol}}_{Z}=\max\{2g-2,0\},
\)
where $g$ is the geometric genus of $Z$. 
\begin{Remark}\label{volume-generically-finite}
Let \(f\colon Z\dashrightarrow Z'\) be a dominant generically finite rational map of degree \(d\) between integral quasi-projective varieties. After passing to smooth projective models, we obtain a generically finite morphism \(\overline f\colon\overline Z\to\overline{Z'}\). The ramification formula gives \(K_{\overline Z}=\overline f^{*}K_{\overline{Z'}}+R\)
for some effective divisor \(R\). Hence
\[
\widetilde{\operatorname{vol}}_{Z}=\operatorname{vol}(K_{\overline Z})\ge\operatorname{vol}(\overline f^{*}K_{\overline{Z'}})=d\,\operatorname{vol}(K_{\overline{Z'}})=d\,\widetilde{\operatorname{vol}}_{Z'}.\]
\end{Remark}

The bridge between degrees with respect to the log-canonical bundle and canonical volumes is provided by the following proposition, which follows from the Arakelov inequality proved by Brunebarbe \cite[Theorem 1.6]{brunebarbe2020increasing}. Brunebarbe's result applies in a more general setting, where Griffiths transversality is a non-trivial condition. In the setting of Shimura varieties, however, this condition is automatically satisfied, and his theorem gives the following Proposition:

\begin{Proposition}\label{Canonical vol est}
Suppose that $|\Nm(\eta)|>4^n$. Then, for every $m$-dimensional subvariety $Z\subset\overline X_1(\eta)$ not contained in $D$, we have
\[\widetilde{\operatorname{vol}}_{Z}
\geq
\frac{1}{n^m}
\operatorname{vol}_{Z}\left(
\left.(K_{\overline X_1(\eta)}-(n-1)D)\right|_Z
\right).
\]
\end{Proposition}

\begin{proof}
Let $\nu:\widetilde Z\to Z$ be a log resolution of $(Z,D|_Z)$ and set
\(
\widetilde D=(\nu^*(D|_Z))_{\mathrm{red}}.
\)
By Proposition~\ref{Arakelov}, we know that
\(
K_{\widetilde Z}+\widetilde D
-\frac1n\nu^*
\left(\left.(K_{\overline X_1(\eta)}+D)\right|_Z\right)
\)
is pseudoeffective. Since
\(
\nu^*(D|_Z)-\widetilde D
\)
is effective, it follows that
\(
K_{\widetilde Z}
-\frac1n\nu^*
\left(\left.(K_{\overline X_1(\eta)}-(n-1)D)\right|_Z\right)
\)
is pseudoeffective. Since the volume does not decrease in pseudoeffective direction we get 
\[
\widetilde{\operatorname{vol}}_{Z}
=
\operatorname{vol}_{\widetilde Z}(K_{\widetilde Z})
\geq
\operatorname{vol}_{Z}\left(
\frac1n\left.(K_{\overline X_1(\eta)}-(n-1)D)\right|_Z
\right).
\]
\end{proof}
The stable base locus of a line bundle $L$ is defined as
\[
\mathbf B(L):=\bigcap_{m\ge1}\operatorname{Bs}(L^{\otimes m}).
\]
Given an ample line bundle $A$, the augmented base locus of $L$ is
\[
\mathbf B_+(L)
:=
\bigcap_{\varepsilon>0}
\mathbf B(L-\varepsilon A),
\]
where $\varepsilon$ ranges over sufficiently small positive rational
numbers. This definition is independent of the choice of $A$. 
\begin{definition}\label{ample mod def} A line bundle $L$ is called \emph{ample modulo} a closed subset $E$ if
\(
\mathbf B_+(L)\subset E.
\)    
\end{definition}

\begin{Lemma}\label{ample rest} Let $Y$ be a projective variety, let $E\subset Y$ be a closed subset, and let $L$ be a line bundle on $Y$. Suppose that $L$ is ample modulo $E.$
Then for every irreducible subvariety $Z\subset Y$ with $Z\not\subset E$, the restricted line bundle $L|_Z$ is big.
 
\end{Lemma}
\begin{proof}  The proof is similar to the proof of \cite[Lemma 4.7]{memariansorkhabi2022positivity}.
Since $Z\not\subset E$ and $\mathbf B_+(L)\subset E$, we have
\(
Z\not\subset \mathbf B_+(L).
\)
Choose an ample line bundle $A$ on $Y$. For $m\gg 0$,
\[
\mathbf B_+(L)=\operatorname{Bs}(mL-A).
\]
Thus $Z$ is not contained in $\operatorname{Bs}(mL-A)$, so there exists a section of $mL-A$ which does not vanish identically on $Z$. Hence
\(
(mL-A)|_Z
\)
has a nonzero section. Therefore
\[
mL|_Z \cong A|_Z+\Sigma
\]
for some effective divisor $\Sigma$ on $Z$. Since $A|_Z$ is ample, it is big. Adding an effective divisor preserves bigness, so $mL|_Z$ is big. Hence $L|_Z$ is big.
\end{proof}

We will use the following positivity result of Bakker-Tsimerman: 
\begin{Proposition}\label{Tsi}(\cite[Theorem E]{bakker2018geometric})\label{PosCan} Suppose that $|\Nm(\eta)|>4^n.$
If $\lambda> 0$ and
\(
|\Nm(\eta)|>
\left(\frac{2\pi\lambda}{n}\right)^{2n},
\)
then the divisor
\(
K_{\Xb(\eta)}+(1-\lambda)D
\)
is ample modulo $D$.
\end{Proposition}

\begin{Lemma}\label{Bigness of small twist}
Suppose $c>n/\pi$ and set
$N(n,c):=1+\left\lfloor(2\pi c/n)^{2n}\right\rfloor$.
If $|\Nm(\eta)|\ge N(n,c)$, then for every subvariety
$Z\subset\Xb(\eta)$ intersecting with $X_1(\eta)$, the divisor
\(
\bigl(K_{\Xb(\eta)}+(1-c)D\bigr)|_Z
\)
is big.
\end{Lemma}

\begin{proof}
By Proposition~\ref{Tsi},
$K_{\Xb(\eta)}+(1-c)D$ is ample modulo $D$. Since $Z\not\subset D$,
the conclusion follows from Lemma~\ref{ample rest}.
\end{proof}

Let \(Z\) be an integral quasi-projective variety over \(k=\mathbb C\),
and set \(K=k(Z)\). The generic fiber \(A_z\) of an abelian scheme
\(A\to Z\) is an abelian variety over \(K\). Conversely, every abelian
variety over \(K\) extends to an abelian scheme over some nonempty open
subset \(U\subset Z\). We denote by
\(
A(Z):=A_z(K)
\)
the Mordell--Weil group of rational sections of \(A\to Z\), and by
\(A(Z)_{\mathrm{tor}}\) its torsion subgroup. 

After replacing \(Z\) by a nonempty open subset, fix the polarization and polarized \(\mathcal O_F\)-module type \(M\) as in Remark~\ref{polarization-data}; thus \(X(1)\) below denotes the corresponding Hilbert modular variety.

Fix a totally real field \(F\). The Hilbert modular variety \(X(1)\)
parametrizes abelian varieties with real multiplication by
\(\mathcal O_F\); see Section~\ref{hilbert}. Accordingly, a family
\(A\to Z\) of such abelian varieties determines a moduli map
\(
\mu_A\colon Z\longrightarrow X(1).
\)
We say that \(A\to Z\) has \emph{no isotrivial directions} if
\(\mu_A\) is generically finite onto its image. When \(Z\) is a curve,
this is equivalent to non-isotriviality, while in higher dimensions it
excludes directions in the base along which the family is constant.
For any ideal \(\eta\subset\mathcal O_F\), the torsion level cover
\(X_1(\eta)\) parametrizes such abelian varieties together with a point
whose annihilator is \(\eta\). A torsion section
\(x\in A(Z)\) with
\(
\operatorname{Ann}_{\mathcal O_F}(x)=\eta
\)
induces a lift
\(
\mu_{A,x}\colon Z\longrightarrow X_1(\eta)
\)
of \(\mu_A\). If \(A\to Z\) has no isotrivial directions, then
\(\mu_{A,x}\) is also generically finite onto its image. Therefore, the canonical volume of $Z$ is greater than or equal to the canonical volume of its image in $X_1(\eta)$; see Remark \ref{volume-generically-finite}, and we will still denote the image of $Z$ in $X_1(\eta)$ by $Z.$

Given our current setup, to prove Theorem~\ref{Canonical vol}, it is enough to prove the following theorem:

\begin{theorem}\label{Main}
Given a real number $v\ge 0$ and a positive integer $n$, there exists an
integer $M=M(n,v)$ such that, for every totally real
field \(F/\mathbb Q\) of degree \(n\), every ideal
\(\eta\subset\mathcal O_F\) satisfying
\(
|\Nm(\eta)|\geq M,
\) and every smooth toroidal compactification
$\Xb(\eta)$ of $X_1(\eta)$, every positive-dimensional subvariety
$Z\subset\Xb(\eta)$ intersecting with $X_1(\eta)$ has canonical volume
$>v$. Explicitly, one may take
\[
M(n,v):=\max\left\{1+\lfloor(4\pi)^{2n}\rfloor,
\left\lceil\left(16n\max\{1,v\}
\bigl(1+n\max\{1,v\}\bigr)\right)^n\right\rceil\right\}.
\]
\end{theorem}

\begin{Remark}\label{explicit N} If $\eta=\operatorname{Ann}_{\mathcal O_F}(x)$, then the order of $x$ is at most $|\Nm(\eta)|$. Theorem~\ref{Main} shows that every torsion point has order less than $M(n,v).$
Since
\(
M(n,v)\le
\left\lceil
\left(4\pi n\max\{1,v\}\right)^{2n}
\right\rceil
\), 
Theorem~\ref{Canonical vol} follows by taking
\[
N(n,v)=
\left(
\left\lceil
\left(4\pi n\max\{1,v\}\right)^{2n}
\right\rceil-1
\right)!.
\]
\end{Remark}

\begin{proof}[Proof of Theorem \ref{Main}]
Let $Z\subset\Xb(\eta)$ be an arbitrary $m$-dimensional subvariety
intersecting with $X_1(\eta)$. Applying Lemma~\ref{Bigness of small twist} with
$c=2n$ shows that, if
$|\Nm(\eta)|\ge N_1:=1+\lfloor(4\pi)^{2n}\rfloor$, then
the divisor \(
\bigl(K_{\Xb(\eta)}-(2n-1)D\bigr)|_Z
\)
is big. Notice that \(N_1>5^n\), so the second part of Proposition~\ref{inj thick} applies. We can write
\[
K_{\Xb(\eta)}-(n-1)D
=
\frac12\bigl(K_{\Xb(\eta)}-(2n-1)D\bigr)
+
\frac12\bigl(K_{\Xb(\eta)}+D\bigr).
\]
Note that $K_{\Xb(\eta)}+D$ is nef on $\Xb(\eta)$ because it is a pull back of an ample bundle from a Baily-Borel compactification; see \cite[Proposition 3.4.b]{mumford1977hirzebruch}. 
Therefore, its restriction to
$Z$ is nef and, by
Theorem~\ref{log deg}, big. By the log-concavity of the volume function
on the big cone of $Z$ (see, for example,
\cite[Theorem~11.4.9]{Lazarsfeld2}), we obtain
\begin{align*}
\operatorname{vol}_Z\left(
\bigl(K_{\Xb(\eta)}-(n-1)D\bigr)|_Z
\right)
&>
\operatorname{vol}_Z\left(
\frac12\bigl(K_{\Xb(\eta)}+D\bigr)|_Z
\right)\\
&=
(2)^{-m}(K_{\Xb(\eta)}+D)^m\cdot Z
\qquad\text{(by nefness)}\\
&\ge
\sinh^{2m}\left(
\frac12\operatorname{arsinh}\left(
\frac12|\Nm(\eta)|^{1/(2n)}
\right)\right)
\qquad\text{(by Theorem~\ref{log deg})}.
\end{align*}
Therefore, Proposition~\ref{Canonical vol est} gives
\[
\widetilde{\operatorname{vol}}_Z
>
n^{-m}
\sinh^{2m}\left(
\frac12\operatorname{arsinh}\left(
\frac12|\Nm(\eta)|^{1/(2n)}
\right)\right).
\]
The constant
\(
N_2:=
\left\lceil
\left(16n\max\{1,v\}
\bigl(1+n\max\{1,v\}\bigr)\right)^n
\right\rceil
\)
ensures that the right-hand side is at least $v$, uniformly for
$1\le m\le n$. Therefore, taking $M=\max\{N_1,N_2\}$ proves the theorem.
\end{proof}

A variety is of general type if and only if its canonical volume is positive. Similar to Theorem \ref{Main} we get the following:  
\begin{Corollary}\label{general type}
Let $F/\mathbb Q$ be a totally real field of degree $n$. 
Every subvariety of $X_1(\eta)$
is of general type provided that 
\(
|\Nm(\eta)|>(2\pi)^{2n}.
\)
\end{Corollary}

\begin{proof}Let $Z\subset\Xb(\eta)$ be an $m$-dimensional subvariety meeting $X_1(\eta)$. Lemma~\ref{Bigness of small twist}, applied with $c=n$, shows that $(K_{\Xb(\eta)}-(n-1)D)|_Z$ is big. Hence, Proposition~\ref{Canonical vol est} gives\[\widetilde{\operatorname{vol}}_Z\ge \frac1{n^m}\operatorname{vol}_Z\left((K_{\Xb(\eta)}-(n-1)D)|_Z\right)>0.\] Therefore $Z$ is of general type.
\end{proof}

\section{Sparsity of rational points}\label{sparsity}

Let \(K\) be a number field. On projective space, we use the absolute
multiplicative height

$$
\operatorname{H}([x_0:\cdots:x_N])
=
\prod_{v\in M_K}
\max_i |x_i|_v^{[K_v:\mathbb Q_v]/[K:\mathbb Q]},
$$
where the absolute values extend the standard ones on \(\mathbb Q\).
This is independent of both the choice of homogeneous coordinates and
the number field containing them. The problem of bounding \(K\)-rational points of bounded height on algebraic varieties is closely related to the determinant method, originating in the work of Bombieri--Pila \cite{bombieri1989number} and further developed by Heath-Brown \cite{heath2002density}. As observed in \cite{ellenberg2023sparsity,brunebarbe2022counting}, uniform lower bounds for the degrees of the subvarieties arising in this process lead to improved bounds for the growth rate of rational points; see also \cite{chiu2025arithmetic} for a mixed-variation setting. In particular, we recall the following theorem of Brunebarbe--Maculan:

\begin{theorem}(\cite[Theorem 4.4]{brunebarbe2022counting})\label{counting}
    Let $E$ be a subvariety of $\mathbb{P}_{K}^{N},$ let $\epsilon>0$ be a real number, let $n\ge 0$ and $e\ge 1$ be integers. Then, there is a real number $C=c(n,e, N, K,E,\epsilon)$ with the following property:
For an  $n$-dimensional subvariety $Y$ of $\Proj^{N}$
 of degree $\le e$ such that each subvariety $Z$ in $Y$ not contained in $E$ has degree $\ge \rho^{\operatorname{dim}(Z)}$ for some integer $\rho \ge 1$, and a real number $B > e^{[K: \Q]\epsilon}$, the following
inequality holds:
$$
\#\{
x\in Y(K)\setminus E \mid \operatorname{H}(x)\le B 
\}
\le
C B^{(1+\epsilon)[K:\Q]n(n+3)/(2\rho)}.$$
\end{theorem}
\begin{Remark}
 There is a typo in the statement of Theorem \ref{counting} in the original paper. However, from its applications and the surrounding discussion in that paper, it is clear that the exponent of \(\rho\) should be the dimension of the subvariety \(Z\), not \(E\).
\end{Remark}

Let $\overline{Y}$ be a projective variety over $K$, let \(E\subset\overline Y\) be a closed subset, and let \(L\) be a semiample line bundle which is ample modulo \(E\); see Definition~\ref{ample mod def}. We fix once and for all an integer \(b\geq1\) such that \(bL\) is base-point free and a basis of its global sections, defining a morphism
$
\varphi_b:\overline Y\longrightarrow\mathbb P_{K}^N.
$
We define
\[
H_L(x):=\operatorname{H}(\varphi_b(x))^{1/b}.
\]
The resulting height depends on the choice of \(b\) and the basis, but if \(H_L\) and \(H'_L\) are heights arising from two such choices, then there exists a constant \(C_0\geq1\) such that
\[
C_0^{-1}H_L(x)\leq H'_L(x)\leq C_0H_L(x)
\]
for every \(x\in\overline{Y}(K)\). We keep the above choice of \(H_L\) fixed throughout. 
\begin{Lemma}\label{semiample counting}
Let \(n=\dim\overline Y\), and suppose that the morphism \(\varphi_b\) identifies \(\overline Y\setminus E\) with the complement of a closed subset \(E'\) in its projective image. Assume that there exists a real number \(\rho>0\) such that
\(L^m\cdot V\geq\rho^m\) for every subvariety \(V\subset\overline Y\) not contained in \(E\), where \(m=\dim V\). 
Then, for every \(\epsilon>0\), there exists a constant \(C_\epsilon>0\) such that
$$
\#\{x\in(\overline Y\setminus E)(K):H_L(x)\leq B\}
\leq C_\epsilon B^{(1+\epsilon)[K:\mathbb Q]n(n+3)/(2\rho)}
$$
for every \(B\geq1\).

\end{Lemma}
\begin{proof}
Compose \(\varphi_b\) with the \(q\)-th Veronese embedding. If \(Z\) is a subvariety of the resulting projective model not contained in \(E'\), and \(V\subset\overline Y\) is its strict transform, then the projection formula gives
$$
\deg(Z)=(qbL)^m\cdot V\geq(qb\rho)^m,
$$
where \(m=\dim Z\). Fix \(\epsilon>0\). For \(q\) sufficiently large, we may choose an integer \(d<qb\rho\) such that
$(1+\epsilon/2)\frac{qb}{d}\leq\frac{1+\epsilon}{\rho};$ for example, one may take \(d=\lceil qb\rho\rceil-1\). Hence \(\deg(Z)>d^{\dim Z}\), and Theorem~\ref{counting}, applied with \(\epsilon/2\), gives

$$
\begin{aligned}
\#\{x\in(\overline Y\setminus E)(K):H_L(x)\leq B\}
&\leq C\bigl(B^{qb}\bigr)^{(1+\epsilon/2)[K:\mathbb Q]n(n+3)/(2d)}\\
&\leq C B^{(1+\epsilon)[K:\mathbb Q]n(n+3)/(2\rho)}.
\end{aligned}
$$

Here we used that the height associated with the Veronese embedding is \(H_L^{qb}\). Once \(q\) is chosen, the projective model is fixed, and hence the resulting constant \(C_\epsilon\) is independent of \(B\). Enlarging \(C\), if necessary, covers the bounded range of \(B\) excluded by Theorem~\ref{counting}.
\end{proof}

\subsection*{Arithmetic model.}
The purpose of this standard discussion is only to place the varieties and line bundles used below over $\mathbb Q$, so that the height argument is well-defined.
We follow the framework of Dimitrov--Tilouine
\cite[\S1]{dimitrovtilouine}, specialized to our setting and notation.
For the levels in Theorem~\ref{log grow}, choose an integral ideal
\(\mathfrak c\) in the narrow ideal class determined by the fixed
polarized type, prime to \(\Nm(\eta)\mathcal O_F\); this is possible by
weak approximation \cite[Remark~1.5]{dimitrovtilouine}. Since \(\eta\)
is relatively prime to \(\Delta_F\), it is also relatively prime to
\(\Nm(\mathfrak c\mathfrak d_F)\mathcal O_F\). Moreover, if
\(\eta\mid(2)\) or \(\eta\mid(3)\), then respectively
\(|\Nm(\eta)|\leq2^n\) or \(|\Nm(\eta)|\leq3^n\), contrary to
\(|\Nm(\eta)|\geq5^n\). Thus both conditions of Dimitrov--Tilouine's
hypothesis \((\mathrm{NT})\) hold.
For every $\mathbb Q$-algebra $R$, consider the determinant map
\[
\det\colon \operatorname{GL}_2(F\otimes_{\mathbb Q}R)
\longrightarrow (F\otimes_{\mathbb Q}R)^\times.
\]
Define
\(
\mathsf G(R)
:=
\left\{
g\in\operatorname{GL}_2(F\otimes_{\mathbb Q}R)
\;\middle|\;
\det(g)\in R^\times
\right\},
\)
where $R^\times$ is embedded diagonally in $(F\otimes_{\mathbb Q}R)^\times$ by $a\mapsto 1\otimes a$. This group is represented by the fiber product
\[
\mathsf G
=
\operatorname{Res}_{F/\mathbb Q}\operatorname{GL}_2
\times_{\operatorname{Res}_{F/\mathbb Q}\mathbb G_m}
\mathbb G_m,
\]
taken with respect to the determinant and the diagonal inclusion. The determinant yields a split exact sequence
\[
1\longrightarrow \operatorname{Res}_{F/\mathbb Q}\operatorname{SL}_2
\longrightarrow \mathsf G
\xrightarrow{\det}\mathbb G_m
\longrightarrow 1,
\]
with section \(a\mapsto \operatorname{diag}(1,a)\).
Since the quotient $\mathbb G_m$ is
commutative and $\operatorname{Res}_{F/\mathbb Q}\operatorname{SL}_2$
is perfect, we have
\(
\mathsf G^{\mathrm{der}}
=
\operatorname{Res}_{F/\mathbb Q}\operatorname{SL}_2.
\)

Let $(\mathsf G,\mathcal H)$ be the associated Hilbert Shimura datum,
and let
\(
\mathcal H^+=\mathbb H^n
\)
be its distinguished connected component. Let $M$ be the rank-two
projective $\mathcal O_F$-module fixed in
Section~\ref{hilbert}, and set
\(
\widehat M:=M\otimes_{\mathbb Z}\widehat{\mathbb Z},
\) and
\(e_1:=(1,0).
\)
Let
\[
K_\eta
:=
\left\{
g\in\mathsf G(\mathbb A_f)
\;\middle|\;
g\widehat M=\widehat M,\quad
g(e_1)\equiv e_1\pmod{\eta\widehat M}
\right\}.
\]
Writing
\(
\mathsf G(\mathbb Q)^+
:=
\{g\in\mathsf G(\mathbb Q):\det(g)>0\},
\)
we have
\[
\mathsf G(\mathbb Q)^+\cap K_\eta=\Gamma_1(\eta),
\qquad
\det\bigl(\mathsf G(\mathbb Q)^+\bigr)=\mathbb Q_{>0}^\times,
\qquad
\det(K_\eta)=\widehat{\mathbb Z}^\times.
\]
Strong approximation for $\mathsf G^{\mathrm{der}}$, together with the
exact sequence above, gives
\[
\mathsf G(\mathbb Q)^+\backslash
\mathsf G(\mathbb A_f)/K_\eta
\overset{\det}{\simeq}
\mathbb Q_{>0}^\times\backslash
\mathbb A_f^\times/\widehat{\mathbb Z}^\times
=
\{1\};
\]
see \cite[Theorem~4.16]{Mil05}.
It follows from the complex uniformization that
\[
\operatorname{Sh}_{K_\eta}(\mathsf G,\mathcal H)(\mathbb C)
\simeq
\bigl(\mathsf G(\mathbb Q)^+\cap K_\eta\bigr)
\backslash\mathcal H^+
=
\Gamma_1(\eta)\backslash\mathbb H^n
=
X_1(\eta).
\]

The defining cocharacter has identical components at all embeddings
$F\hookrightarrow\mathbb R$, so its conjugacy class is Galois
invariant and the reflex field of $(\mathsf G,\mathcal H)$ is
$\mathbb Q$. Canonical-model theory therefore endows $X_1(\eta)$
with a model over $\mathbb Q$; see
\cite[\S\S12--14]{Mil05}. Moreover, by
\cite[Theorem~12.3]{Pin90}, its Baily--Borel compactification
$X_1(\eta)^*$ is a normal projective variety over $\mathbb Q$.
We use these models throughout this section.

\subsection*{Counting result}
Let \(\overline X_1(\eta)\) be a smooth toroidal compactification fixed in Section \ref{hilbert}, with boundary divisor \(D\), and consider the log-canonical bundle on \(\overline X_1(\eta)\):
\[
L:=K_{\overline X_1(\eta)}+D.
\]
Let \(\pi\colon\overline X_1(\eta)\longrightarrow X_1(\eta)^*
\) be the Baily--Borel contraction. By
\cite[Theorem~5.2]{dimitrovtilouine}, there exist
\(b\geq 1\) and a very ample line bundle \(A\) on \(X_1(\eta)^*\), defined
over \(\mathbb Q\), such that
\(
\pi^*A\simeq bL.
\)
Fix a closed immersion
\(
\iota\colon X_1(\eta)^*
\hookrightarrow
\mathbb P^N_{\mathbb Q}
\)
defined by $A$, and set
\[
H_A(x):=H\bigl(\iota(x)\bigr),
\qquad
H_L(x):=H_A(x)^{1/b}
\]
for $x\in X_1(\eta)(\overline{\mathbb Q})$. We keep these choices
fixed. Since $A$ is very ample, $H_L$ has the Northcott
property on $X_1(\eta)(K)$ for every number field $K$. For further properties of this height see \cite[\S 8]{memariansorkhabi2025volumes}.

\begin{theorem}\label{log grow}
Assume that \(|\Nm(\eta)|\geq5^n\) and that \(\eta\) is relatively prime to the discriminant \(\Delta_F\).
For every number field \(K\), set
\[
N_{\eta,K}(B)
=\#\{x\in X_1(\eta)(K):H_L(x)\leq B\}.
\]
Then
\[
\limsup_{B\to\infty}
\frac{\log\max\left\{1,N_{\eta,K}(B)\right\}}
{\log B}
\leq \delta_{\eta,K,n},
\]
where
\[
\delta_{\eta,K,n}
=
\frac{\sqrt5[K:\mathbb Q]n(n+3)}
{|\Nm(\eta)|^{1/(2n)}}.
\]
\end{theorem}

\begin{Remark}[\textbf{Sparsity of rational points}]
In Theorem \ref{log grow}, \(\delta_{\eta,K,n}\to0\) uniformly when \(n\) and
\([K:\mathbb Q]\) are bounded and
\(|\Nm(\eta)|\to\infty\) through ideals prime to \(\Delta_F\). Therefore, as we go higher in the tower \(X_1(\eta)\), its \(K\)-rational points become sparser.
\end{Remark}

\begin{proof}
Set
\(
s_\eta
:=
\sinh^2\left(
\frac12\operatorname{arsinh}
\left(\frac12|\Nm(\eta)|^{1/(2n)}\right)
\right)\)
and
\( E:=X_1(\eta)^*\setminus X_1(\eta).
\)
Let \(Z\subset X_1(\eta)^*\) be an \(m\)-dimensional subvariety not contained in \(E\), with \(m>0\), and let \(\widetilde Z\subset\overline X_1(\eta)\) be the strict transform of \(Z_{\mathbb C}\). Since \(\pi\) is an isomorphism over \(X_1(\eta)\), the projection formula and Theorem~\ref{log deg} give
\[
\deg_A(Z)
=
\bigl(\pi^*A\bigr)^m\cdot\widetilde Z
=
b^mL^m\cdot\widetilde Z
\geq
(2bs_\eta)^m.
\]
 Applying Lemma~\ref{semiample counting} to \(X_1(\eta)^*\), \(E\), and \(A\), with \(\rho=2bs_\eta\), and using \(H_L=H_A^{1/b}\), we obtain, for every \(\epsilon>0\),
\[
N_{\eta,K}(B)
\leq
C_\epsilon B^{(1+\epsilon)[K:\mathbb Q]n(n+3)/(4s_\eta)}.
\]
It follows that
\[
\limsup_{B\to\infty}
\frac{\log\max\{1,N_{\eta,K}(B)\}}{\log B}
\leq
\frac{[K:\mathbb Q]n(n+3)}{4s_\eta}.
\]
Finally,
\(
4s_\eta
\geq
\frac{1}{\sqrt5}|\Nm(\eta)|^{1/(2n)},
\)
since \(|\Nm(\eta)|^{1/(2n)}\geq\sqrt5\), which gives the stated bound.
\end{proof}

 An important arithmetic ingredient in both
\cite{ellenberg2023sparsity,brunebarbe2022counting} is an integral
Chevalley--Weil argument. Namely, an integral point determines an
unramified torsor arising from the fiber of a finite \'etale cover;
Hermite--Minkowski implies that only finitely many such torsors occur,
and twisting by these torsors allows every integral point to lift to a
rational point on one of finitely many twists. We conclude this section
with a remark explaining why the same argument does not directly apply
to rational points on quasi-projective varieties, and hence why we adopt
an intrinsic approach in this paper.

\begin{Remark}[\textbf{Integral versus Rational Chevalley--Weil
}]\label{Chevalley--Weil}
Let \(X\) be a quasi-projective variety over a number field \(K\), and
let \(f\colon Y\to X\) be a finite \'etale Galois cover with Galois
group
\(
G:=\operatorname{Aut}_X(Y).
\)
By spreading out, there exist a finite set \(S\) of places of \(K\) and
\(\mathcal O_{K,S}\)-models \(\mathcal X\) and \(\mathcal Y\) of \(X\)
and \(Y\) such that \(f\) extends to a finite \'etale \(G\)-torsor
\(
\mathcal Y\to\mathcal X,
\)
where \(G\) is regarded as a constant finite group scheme. For every
\(x\in\mathcal X(\mathcal O_{K,S})\), the fiber
\[
\mathcal Y_x:=
\mathcal Y\times_{\mathcal X}\operatorname{Spec}\mathcal O_{K,S}
\]
defines a class in \(H^1(\mathcal O_{K,S},G)\). By
Hermite--Minkowski, this set is finite; see
\cite[Proposition~4.6]{brunebarbe2022counting} or
\cite[Proof of Claim~2]{ellenberg2023sparsity}. Twisting \(Y\to X\)
by representatives of these finitely many classes therefore produces
finitely many covers \(Y_i\to X\) such that every \(S\)-integral point
of \(X\) lifts, in fact integrally, to some \(Y_i\). This is the form
of the Chevalley--Weil argument used in
\cite[\S 3]{ellenberg2023sparsity} and
\cite[Proposition~4.12]{brunebarbe2022counting}.

For rational points on a nonproper variety, however, the corresponding
torsors need not be unramified outside any fixed finite set of places.
For example, under the finite \'etale morphism
\[
\mathbb G_m\longrightarrow\mathbb G_m,\qquad t\longmapsto t^2,
\]
the fiber above \(a\in K^\times\) is the torsor defined by \(t^2=a\).
Already for \(K=\mathbb Q\) and \(a=p\), where \(p\) is an odd prime,
the field \(\mathbb Q(\sqrt p)\) is ramified at \(p\). Therefore, as \(p\)
varies, the resulting torsor classes are ramified at infinitely many
places and are pairwise distinct. By contrast, if
\(a\in\mathcal O_{K,S}^{\times}\) and \(S\) contains every finite place
of \(K\) above \(2\), then
\[
\operatorname{Spec}\mathcal O_{K,S}[t]/(t^2-a)
\longrightarrow\operatorname{Spec}\mathcal O_{K,S}
\]
is finite \'etale, since its discriminant \(4a\) is a unit. Hence the
associated torsor is unramified outside \(S\).
\end{Remark}

\bibliographystyle{alpha}
\bibliography{Ref}

\end{document}